\documentclass[12pt]{amsart}
\usepackage{amsmath,amssymb,amsthm}
\usepackage{setspace}
\usepackage{moreverb}
\usepackage{mathrsfs}
\usepackage{url}
\usepackage[all]{xy}
\usepackage{lscape}
\usepackage{tikz}
\usetikzlibrary{arrows}
\usepackage{color}

\usepackage[ top=1in, bottom=1in, left=1.25in, right=1.25in ]{geometry}
\usepackage[colorlinks=true, linkcolor=blue, citecolor=green, urlcolor=cyan, pagebackref]{hyperref}

\newcommand{\ignore}[1]{}

\newtheorem{theorem}{Theorem}[section]
\newtheorem{lemma}[theorem]{Lemma}
\newtheorem{corollary}[theorem]{Corollary}
\newtheorem{proposition}[theorem]{Proposition}
\newtheorem{algorithm}[theorem]{Algorithm}
\theoremstyle{definition}
\newtheorem{definition}[theorem]{Definition}
\newtheorem{example}[theorem]{Example}

\theoremstyle{remark}
\newtheorem{remark}[theorem]{Remark}

\numberwithin{equation}{section}

\input xy
\xyoption{all}

\newcommand{\bP}{{\mathbb{P}}}

\newcommand{\bC}{\mathbb{C}}

\newcommand{\bZ}{{\mathbb{Z}}}

\newcommand{\bQ}{{\mathbb{Q}}}

\definecolor{grey}{rgb}{0.75,0.75,0.75}
\definecolor{orange}{rgb}{1.0,0.5,0.5}
\definecolor{brown}{rgb}{0.5,0.25,0.0}
\definecolor{pink}{rgb}{1.0,0.5,0.5}
\newcommand{\adj}{\mathrm{Adj}}

\newcommand{\fM}{{\mathfrak m}}

\newcommand{\fa}{\mathfrak{a}}
\newcommand{\fb}{\mathfrak{b}}

\newcommand{\cO}{{\mathcal O}}

\newcommand{\lra}{{\longrightarrow}}

\newcommand{\A}{\mathfrak{a}}

\newcommand{\Z}{\mathbb{Z}}

\newcommand{\R}{\mathbb{R}}

\newcommand{\Oc}{\mathcal{O}}

\newcommand{\J}{\mathcal{J}}

\begin{document}

\title[Multiplier ideals in two-dimensional local rings]{Multiplier ideals in two-dimensional local rings with rational singularities }

\author[M. Alberich-Carrami\~nana]{Maria Alberich-Carrami\~nana}

\author[J. \`Alvarez Montaner]{Josep \`Alvarez Montaner}

\author[F. Dachs-Cadefau]{Ferran Dachs-Cadefau }

\address{Departament de Matem\`atiques\\
Universitat Polit\`ecnica de Catalunya\\ Av. Diagonal 647, Barcelona
08028, Spain} 
%\address{Institut de Rob\`otica i Inform\`atica Industrial\\ CSIC-UPC \\ Llorens i Artigues 4-6, Barcelona 08028, Spain}
\email{Maria.Alberich@upc.edu, Josep.Alvarez@upc.edu}

\address{KU Leuven\\ Department of Mathematics \\ Celestijnenlaan 200B box 2400\\
BE-3001 Leuven, Belgium}

\email{Ferran.DachsCadefau@wis.kuleuven.be}

\thanks{All three authors were partially supported by Generalitat de Catalunya 2014 SGR-634 project  and
Spanish Ministerio de Econom\'ia y Competitividad
MTM2012-38122-C03-01/FEDER. FDC  is also supported by the
KU Leuven grant OT/11/069. MAC is also with the  Institut de Rob\`otica i Inform\`atica Industrial (CSIC-UPC) and the Barcelona Graduate School of Mathematics (BGSMath).}

%\keywords {Multiplier ideals, jumping numbers, plane curves.}

%\subjclass[2000]{Primary 13D45, 13N10}

\begin{abstract}
The aim of this paper is to study jumping numbers and
 multiplier ideals of any
ideal in a two-dimensional local ring with a rational singularity.
In particular we reveal which information encoded in a multiplier
ideal determines the next jumping number. This leads to an algorithm
to compute sequentially the jumping numbers and the whole chain of
multiplier ideals in any desired range. As a consequence of our
method we develop the notion of {\it jumping divisor} that allows to
describe the jump between two consecutive multiplier ideals. In
particular we find a unique minimal jumping divisor that is studied
extensively.
\end{abstract}

\maketitle

\section{Introduction}

Let $X$ be a complex algebraic variety with mild singularities and
$\cO_{X,O}$ the local ring of a point $O\in X$. To any ideal $\fa
\subseteq \cO_{X,O}$ one may associate a family of {\it multiplier
ideals} $\J(\A^\lambda)$ parametrized by positive rational numbers
$\lambda\in \bQ_{>0}$. Indeed, they form a nested sequence of ideals
$$\Oc_{X,O}\varsupsetneq\J(\A^{\lambda_1})\varsupsetneq\J(\A^{\lambda_2})
\varsupsetneq...\varsupsetneq\J(\A^{\lambda_i})\varsupsetneq...$$
and the rational numbers $0< \lambda_1 < \lambda_2 < \cdots$ where
the multiplier ideals change are called \emph{jumping numbers}. The
first jumping number $\lambda_1$ is also known as the {\it
log-canonical threshold}. Multiplier ideals and their associated
jumping numbers have proven to be a powerful tool to understand the
geometry of singularities. They are defined using a log-resolution
of the pair $(X,\fa)$. In fact, smaller or more dense jumping
numbers can be thought to correspond to ``worse`` singularities.

\vskip 2mm

The aim of this paper is to present a new approach to the
understanding of multiplier ideals and jumping numbers of any ideal
$\fa$ in the local ring $\cO_{X,O}$ of a complex surface $X$ having
at worst a rational singularity at $O$. This is a case, especially
when $X$ is smooth, that has received a lot of attention in recent
years because of the interesting properties these invariants satisfy
(see the works of Favre-Jonsson \cite{FJ04}, \cite{FJ05},
Lipman-Watanabe\cite{LW} or Tucker \cite{Tuc09}). This is also one
of the few cases where explicit computations have been done.

\vskip 2mm

For simple complete ideals or irreducible plane curves in a smooth
surface, J\"arviletho \cite{Jar11} and Naie \cite{Nai09} provide a
closed formula for the set of jumping numbers in terms of some
invariants of the singularity,  the {\it Zariski exponents}. To give
a closed formula for any general ideal is beyond the scope of this
work. A formula for the log-canonical threshold already becomes
quite complicated as one may see in the papers of Kuwata
\cite{Kuw99} and Galindo-Hernando-Monserrat \cite{GHM12}.

\vskip 2mm

For the case of any ideal in a surface with a rational singularity
we must refer to the work of Tucker \cite{Tuc10} where he gives a
simple algorithm to compute the set of jumping numbers. To such
purpose, he developed the notion of divisors that {\it (critically)
contribute}, building upon previous work of Smith-Thompson
\cite{ST07}.  We may interpret jumping numbers as being parametrized
by contributing divisors and critical divisors are more economic to
detect
 since the complete ideals they define are very close
to their corresponding multiplier ideal.  The algorithm he proposes
uses a characterization of critical divisors that allows them to be
found and consequently allows the corresponding jumping numbers to
be computed.

\vskip 2mm

A similar strategy is used by Hyry-J\"{a}rvilehto in \cite{HJ11}
where they proved that jumping numbers are parameterized by more
general complete ideals\footnote{Contributing divisors describe
complete ideals nested in between consecutive multiplier ideals. The
ideals considered in \cite{HJ11} are not necessarily nested.}.
Moreover, they provide a combinatorial criterion to detect a
suitable ideal and its corresponding jumping number.

\vskip 2mm

The aim of this paper is to understand the whole change between a
multiplier ideal to the next one, and to reveal what information
encoded in a multiplier ideal determines the next jumping number.
This is done in our main result Theorem \ref{lct_ideal} and it gives
rise to an algorithm to compute the ordered sequence of multiplier
ideals in any desired range of the real line. The algorithm avoids
considering candidates and computes sequentially at each step a
jumping number and its associated multiplier ideal. This new
algorithm improves in efficiency the computation of jumping numbers
when compared with Tucker's algorithm.

\vskip 2mm

Perhaps the most important contribution of our method lies in
finding a divisor, that we name the \emph{minimal jumping divisor},
tightly related to the aforementioned algorithm, which enables one to
obtain a multiplier ideal from the previous one, and vice versa.
This jumping divisor is studied, in particular its geometric
structure on the dual graph, and it is compared with the previously
known critically contributing divisors.

\vskip 2mm

The structure of this paper is as follows: In Section \S\ref{preli} we
introduce the basics of the theory of multiplier ideals and some of
the tools in the theory of singularities that we will need in the
rest of the paper. We pay special attention to the equivalence
between complete ideals and antinef divisors developed by Lipman in
\cite{Lip69} since this is the way we will present multiplier
ideals. In particular we provide a new method to compute the antinef
closure of any given divisor, generalizing previous versions of
Casas-Alvero \cite{Cas00} and Reguera \cite{Reg96}.

\vskip 2mm

In Section \S\ref{main_algorithm} we present the main result of this paper in Theorem
\ref{lct_ideal}. It gives a generalization of a well-known formula
for the log-canonical threshold and allows us to compute a jumping
number from the data given by the preceding a multiplier ideal. This
leads to the desired algorithm that computes sequentially the chain
of multiplier ideals.

\vskip 2mm

In Section \S\ref{jumping_divisor} we develop the theory of {\it jumping divisors} that
allows us to describe the whole jump between two consecutive
multiplier ideals. Quite surprisingly, the algorithm we develop in
Section \S\ref{main_algorithm} allows us to construct the unique {\it minimal jumping
divisor} associated to every jumping number. It is minimal in the
sense that no proper subdivisor gives the jump between consecutive
multiplier ideals. Moreover, we prove in Theorem \ref{invariance}
that  minimal jumping divisors are {\it generically} invariant with
respect to log-resolutions of the ideal and they satisfy some nice
geometric properties when viewed in the dual graph.

\vskip 2mm

Finally, in Section \S\ref{mini} we present the theory of jumping divisors in
a more general framework that we develop using the results of
Hyry-J\"{a}rvilehto \cite{HJ11} and their relation with the theory
of contributing divisors of Tucker  \cite{Tuc10}.  The main result
of this section is the fact that, among all the contributing
divisors associated to a jumping number that give the same ideal,
there is a minimal one. For example, critical divisors are of this
type. It turns out that these minimal contributing divisors are all
contained in the minimal jumping divisor and inherit the same
invariance property with respect to log-resolutions of the ideal.

\vskip 2mm

{\bf Acknowledgments:} We wish to thank V\'ictor Gonz\'alez Alonso for uncountable discussions
that we had with him during the realization of this work. The authors would also like to 
thank Pierrette Cassou-Nogu\' es and Wim Veys for the comments received and Manuel Gonz\'alez-Villa
for a careful reading of a previous version of the manuscript.

%%%%%%%%%%%%%%%%%%%%%%%%%%%%%%%%%%%%%%%%%%%%%%%%%%%%%%%%%%%%%%%%%%%%%%%%%%%%%%%%%%%%%%%%%%%%%%%%%%

\section{Preliminaries} \label{preli}
Let $X$ be a normal surface and $O$ a point where $X$ has at worst a
rational singularity. That is, there exists a desingularization
$\pi: X' \rightarrow X$ such that the stalk at $O$  of the higher
direct image $R^1 \pi_* \cO_{X'}$ is zero. This property is then
satisfied for any desingularization. The theory of rational
singularities was introduced by Artin in \cite{Art66} and further
developed by Lipman in \cite{Lip69}. Another reference that we will
follow closely is  \cite{Reg96}.

\vskip 2mm

Let $\fa \subseteq \cO_{X}$ be an ideal sheaf. Denote $\fM =
\fM_{X,O} \subseteq \cO_{X,O}$ the maximal ideal of the local ring
$\cO_{X,O}$ at $O$. Throughout this work we will often consider the
case where $\fa \subseteq \fM$ is an $\fM$-primary ideal, which can
be identified with an ideal sheaf that equals $\cO_X$ outside the
point $O$ (we will use both languages interchangeably, depending on
the context). Recall that a {\em log-resolution} of the pair
$\left(X,\fa\right)$ (or of $\fa$, for short) is a proper birational
morphism $\pi: X' \rightarrow X$ such that
\begin{enumerate}
\item $X'$ is smooth,
\item the preimage of $\fa$ is locally principal, that is, $\fa\cdot\cO_{X'} = \cO_{X'}\left(-F\right)$
for some effective Cartier divisor $F$, and
\item $F+E$ is a divisor with simple normal crossings support where $E = Exc\left(\pi\right)$ is the exceptional locus.
\end{enumerate}

From now on, consider a given log-resolution of $\fa$. Since the
point $O$ has (at worst) a rational singularity, the exceptional
locus $E$ is a tree of smooth rational curves $E_1,\dots,E_r$.
Furthermore, the matrix of intersections $\left(E_i\cdot
E_j\right)_{1\leqslant i,j \leqslant r}$ is negative-definite.

\vskip 2mm

Let ${\rm Div}(X')$ be the group of integral divisors in $X'$, i.e.
divisors of the form $D = \sum_i d_i E_i$ where the $E_i$ are
pairwise different (non necessarily exceptional) prime divisors and
$d_i \in \bZ$. Among them, we will consider divisors in the lattice
$\Lambda:= \bZ E_1 \oplus \cdots \oplus \bZ E_r$ of exceptional
divisors and we will simply refer them as divisors with {\it
exceptional support}. Any divisor $D\in {\rm Div}(X')$ has  a
decomposition $D=D_{\rm exc} + D_{\rm aff}$ into its  {\it
exceptional} and  {\it affine} part\footnote{We follow the
terminology of Lipman-Watanabe \cite{LW}} according to its support.
Our main example is the divisor $F$ such that $\fa\cdot\cO_{X'} =
\cO_{X'}\left(-F\right)$. In this case we will denote its
exceptional and affine  part as
$$ F_{\rm exc}= \sum_{i=1}^r e_i E_i \hskip 5mm {\rm and} \hskip 5mm  F_{\rm aff}= \sum_{i=r+1}^s e_i E_i$$
where, by definition, the $e_i$ are non-negative integers. Whenever
$\fa$ is an $\fM$-primary ideal, the divisor $F$ is just supported
on the exceptional locus. i.e. $F=F_{\rm exc}$.

\vskip 2mm

\begin{remark}
 Let $C:f=0$ be a curve
defined by an element $f \in \cO_{X,O}$. The \emph{total transform}
of $C$ is the pull-back $\overline{C}:=\pi^*C$ and its \emph{strict
transform} $C'$ is the closure of $\pi^{-1}(C-\{O\})$. The total
transform has  a
 presentation $\overline{C}=C'+\overline{C}_{\rm exc}=C'+\sum d_iE_i$ where the weights $v_i(f):=d_i$ are
 the \emph{values}  of the curve $C$ at $E_i$.  Recall that $f\in \fa$ whenever ${C'}+\overline{C}_{\rm exc} \geqslant F$ and
 $f$ is {\it generic} in $\fa$  if  $\overline{C}_{\rm exc}=F_{\rm exc}$ and $C'- F_{\rm aff}$ has no singular points.
\end{remark}

\vskip 2mm

More generally, we will also consider $\bQ$-divisors in ${\rm
Div}_{\bQ}(X')= {\rm Div}(X')\otimes_{\bZ} \bQ$ or divisors in the
$\bQ$-vector space $\Lambda_\bQ:= \bQ E_1 \oplus \cdots \oplus \bQ
E_r$. The main example  will be the {\it relative canonical divisor}
$K_{\pi}$. Indeed, the definition of $K_{\pi}$ is quite subtle if
$O$ is singular, because at first sight one can only define a
canonical divisor $K_X$ of $X$ as a Weil divisor. Since rational
singularities are in particular $\bQ$-factorial, there exists a
positive integer $m$ such that $mK_X$ is Cartier, which can be
pulled back to $X'$ and allows to define $K_{\pi} = K_{X'} -
\frac{1}{m}\pi^*\left(mK_X\right)$. Alternatively,  $$ K_{\pi}=
\sum_{i=1}^r k_i E_i$$ is supported on the exceptional locus $E$,
and must satisfy
\begin{equation} \label{eq-relative-canonical}
\left(K_{\pi}+E_i\right)\cdot E_i = \left(\sum_{j=1}^r k_j E_j \cdot
E_i\right) + E_i^2 = -2
\end{equation}
for every exceptional component $E_i$ because of the adjunction
formula. This property indeed characterizes $K_{\pi}$ because the
intersection form on $E$ is negative-definite, and therefore the
system defined by equations (\ref{eq-relative-canonical}) has a
unique solution $\left(k_1,\ldots,k_r\right)$. However, the $k_i$
are not necessarily integral, and can even be negative. In the case
that $k_i > -1$ (resp. $k_i \geqslant -1$) for all $E_i$, one says
that $X$ has a {\it log-terminal singularity} (resp. {\it
log-canonical singularity}) at $O$.

\vskip 2mm

For any $\bQ$-divisor $D = \sum_i d_i E_i \in {\rm Div}_{\bQ}(X')$,
we denote its {\em round-down} and {\em round-up} as
$$\left\lfloor D \right\rfloor = \sum_i \left\lfloor d_i \right\rfloor E_i \quad \text{and} \quad \left\lceil D \right\rceil = \sum_i \left\lceil d_i \right\rceil E_i.$$
The {\em fractional part} of $D$ is then $\left\{D\right\} =
D-\left\lfloor D \right\rfloor = \sum_i \left\{d_i\right\} E_i$. In
the sequel we will denote the value of each component $E_i$ of $D$
as $v_{E_i}(D):=d_i$ . If no confusion arises, we will simply denote
the value of the components as $v_i(D):=v_{E_i}(D)$.

\vskip 2mm

\subsection{Dual graph}

The combinatorics of the log-resolution of $\fa$ can be encoded
using the so-called {\it dual graph}. This is a rooted tree where
the vertices represent the irreducible components $E_i\leqslant F$
and two vertices are joined by an edge if the corresponding divisors
intersect.

\vskip 2mm

Given any component $E_i$, we will denote by $\adj\left(E_i\right)$
the set of  components $E_j$, $j\neq i$, sharing an edge with $E_i$,
i.e. $E_i \cdot E_j = 1$, and by
$$a\left(E_i\right) = \#\adj\left(E_i\right) = E_i \cdot \left(F^{\rm red}-E_i\right)$$
the number of such components which is the {\it valence} of the
vertex representing $E_i$, where $F^{\rm red}$ denotes de reduced
divisor with the same support as $F$. An {\it end} of the dual graph
is nothing but a vertex with {valence} $1$, i.e. a vertex $E_i$ such
that $a\left(E_i\right)=1$. More generally, for any effective
subdivisor $D=E_{i_1}+\cdots+E_{i_m} \leqslant F$ we define
$$\adj_D\left(E_i\right) = \left\{E_j \leqslant D \hskip 2mm | \hskip 2mm E_i \cdot E_j = 1\right\}$$
and $a_D\left(E_i\right) = \#\adj_D\left(E_i\right)$. We denote by
$v_D = m$ (resp. $a_D$) the number of components of $D$ (resp. the
number of intersections between two components of $D$). Since the
dual graph is a tree it is clear that $$\sum_{E_i \leqslant
D}a_D\left(E_i\right) = 2 a_D$$ and that $v_D - a_D$ equals the
number of connected components of $D$. An end of the subgraph
associated to $D$ is a vertex with valence $1$ or $0$. The later
meaning that $E_i$ is an isolated component of $D$.

\vskip 2mm

For any exceptional component $E_i$, we define the {\em excess} (of
$\fa$) at $E_i$ as $\rho_i = - F \cdot E_i$. It can be interpreted
as the number of branches of the strict transform of a curve defined
by a generic element $f \in \fa$ that intersect the component $E_i$.
 Indeed, if its total transform is $\overline{C}=C'+F$ then $0 = \overline{C} \cdot E_i = C' \cdot E_i +
F \cdot E_i = C'\cdot E_i - \rho_i$, which proves the claim.

\vskip 2mm

There are two kinds of exceptional divisors that will play a special
role:
\begin{enumerate}
\item[$\bullet$] A component $E_i$ of $E$ is a {\em rupture} component if $a\left(E_i\right) \geqslant 3$,
that is, it intersects at least three more components of $E$
(different from $E_i$).

\item[$\bullet$] We say that $E_i$ is {\em dicritical} if $\rho_i > 0$. By \cite{Lip69},  dicritical components correspond to
{\it Rees valuations}.
\end{enumerate}

We also mention that non-exceptional components also correspond to
{\it Rees valuations}.

\subsection{Complete ideals and antinef divisors}\label{unloading}
Given an effective $\bQ$-divisor $D=\sum d_i E_i \in {\rm
Div}_{\bQ}(X')$  we may consider its associated (sheaf) ideal
$\pi_{\ast}\Oc_{X'}(-D):= \pi_{\ast}\Oc_{X'}(-\lceil D\rceil)$. Its
stalk at $O$ is
$$I_D:=\{ f\in \Oc_{X,O} \hskip 2mm | \hskip 2mm v_i(f)\geqslant \lceil d_i \rceil \hskip 2mm \text{for all} \hskip 2mm  E_i \leqslant D \}.$$
This is a complete ideal of $\Oc_{X,O}$ that is $\fM$-primary
whenever $D$ has exceptional support, i.e. $D\in \Lambda_\bQ$. Any
two divisors $D,D'\in {\rm Div}_{\bQ}(X')$ defining the same
complete ideal $\pi_{\ast}\Oc_{X'}(-D)= \pi_{\ast}\Oc_{X'}(-{D'})$
are called {\it equivalent divisors}.

\vskip 2mm

In the equivalence class of a given divisor one may find a unique
maximal representative. First, recall that an effective divisor with
integral coefficients $D \in {\rm Div}(X')$   is called
\emph{antinef} if $-D\cdot E_i \geqslant 0$,  for every exceptional
prime divisor $E_i$. It is worth to point out that the affine part
of  $D=D_{\rm exc} + D_{\rm aff}$ satisfies $ D_{\rm aff}\cdot E_i
\geqslant 0$. Therefore $D$ is antinef whenever $- D_{\rm exc}\cdot
E_i \geqslant D_{\rm aff}\cdot E_i$.

\vskip 2mm

In the work of Lipman (see \cite[\S 18]{Lip69}) one may find the
following correspondence that we will heavily use throughout this
work.

\begin{theorem}
There is a one to one correspondence between antinef divisors in
${\rm Div}(X')$ and complete ideals in $\Oc_{X,O}$.  In particular,
antinef divisors in $\Lambda$ correspond to $\fM$-primary complete
ideals.
\end{theorem}

In order to find the representative in the equivalence class of a
given divisor $D\in {\rm Div}_{\bQ}(X')$ we will consider its
so-called {\it antinef closure} $\widetilde{D}$. The existence of
such divisor is a consequence of the following results that can be
found in \cite[\S 18]{Lip69}, but we also refer to \cite{Tuc09} and
\cite{LW} for more insight.

\begin{lemma} \label{def_antinef}
For any effective $\bQ$-divisor $D\in {\rm Div}_{\bQ}(X')$ there
exists a unique minimal integral antinef divisor $\widetilde{D} \in
{\rm Div}(X')$ satisfying $\widetilde{D}\geqslant D$ that is called
the \emph{antinef closure} of $D$. In particular, any antinef
divisor $D'$ such that $D'\geqslant D$ must satisfy $D'\geqslant
\widetilde{D}\geqslant D$
\end{lemma}

\begin{proposition}\label{equiv_antinef}
An effective $\bQ$-divisor $D\in {\rm Div}_{\bQ}(X')$ and its
antinef closure $\widetilde{D}\in {\rm Div}(X')$ are equivalent,
i.e. $$\pi_{\ast}\Oc_{X'}(-D)= \pi_{\ast}\Oc_{X'}(-\widetilde{D}).
$$
\end{proposition}

One of the advantages of working with antinef divisors is that they
provide the following characterization for the inclusion (or strict
inclusion) of two given complete ideals.

\begin{proposition}\label{semicont_antinef}
Let $D_1,D_2$ be two antinef divisors in ${\rm Div}(X')$. Then:

\begin{itemize}

\item[i)] $\pi_* \Oc_{X'}(-D_1)\supseteq \pi_* \Oc_{X'}(-D_2)$ if and only
if ${D_1}\leqslant D_2$.

\item[ii)] $\pi_* \Oc_{X'}(-D_1) \varsupsetneq \pi_* \Oc_{X'}(-D_2)$ if and only if
${D_1} < D_2$.
\end{itemize}

\end{proposition}

For non-antinef divisors we can only claim the following
implication:

\begin{proposition}\label{prop_clausura}
Let $D_1,D_2$ be two  divisors in ${\rm Div}_{\bQ}(X')$ such that
$D_1\leqslant D_2$. Then:

\begin{itemize}

\item[i)] $\pi_* \Oc_{X'}(-D_1) \supseteq \pi_* \Oc_{X'}(-D_2)$.

\item[ii)] $\widetilde{D_1}\leqslant \widetilde{D_2}$.
\end{itemize}
The converses to these results are no longer true.

\end{proposition}

\vskip 2mm

In general, the divisors that will be considered in this work are
not antinef. In order to compute their antinef closure we will use
an inductive procedure called {\it unloading} that was already
described in the work of Enriques \cite[IV.II.17]{EC15} (for the
case of smooth varieties) and Laufer's procedure to compute the
fundamental cycle \cite{Lau72} (for varieties with rational
singularities). Here we will present a new version that is a
generalization of both the unloading procedures described by
Casas-Alvero \cite[\S 4.6]{Cas00} (for smooth varieties) and Reguera
\cite{Reg96} (for varieties with rational singularities).

\vskip 2mm

{\bf Unloading procedure:}   Let $D\in {\rm Div}_{\bQ}(X')$ be any
$\bQ$-divisor. Its {\it excess} at the exceptional prime divisor
$E_i$ is the integer $\rho_i= -\lceil D \rceil \cdot E_i$. Denote by
$\Theta$ the set of exceptional components $E_i \leqslant D$ with
negative excesses, i.e. $$\Theta:= \{E_i \leqslant D_{\rm exc}
\hskip 2mm | \hskip 2mm  \rho_i= -\lceil D \rceil \cdot E_i <0 \}.$$
To {\it unload values} on this set is to consider the new divisor
$$D'= \lceil D \rceil + \sum_{E_i \in \Theta} n_i E_i,$$ where $n_i=
\left \lceil \frac {\rho_i}{E_i^2} \right \rceil$. Notice that $n_i$
is the least integer number such that
$$(\lceil D\rceil+n_i E_i)\cdot E_i= -\rho_i + n_i E_i^2 \leqslant 0.$$

\begin{remark}
Casas-Alvero considered at each step just one component with
negative excess. Reguera also considered one component with negative
excess but in her case she also imposed $n_i=1$ at each step. In
this sense, our approach is more economic from a computational point
of view. Furthermore, our procedure allows unloading on divisors
with affine part\footnote{Our method also differs from the one
considered by Lipman-Watanabe in \cite{LW}.}, which will enable us
to treat in a unified way multiplier ideals of both curves and not
necessarily  $\fM$-primary complete ideals.
\end{remark}

The correctness of the unloading procedure is a consequence of the
following results.

\begin{proposition}
Let $D'$ be the divisor obtained from a divisor $D\in {\rm
Div}_{\bQ}(X')$ after one single unloading step. Then $I_{D'}=I_D$.
\end{proposition}

\begin{proof}
It is clear from its construction that  $I_{D'}\subseteq I_D$. Pick
$f\in I_D$ and let $\overline{C}=C'+\overline{C}_{\rm exc}$ be the
total transform of the curve $C$ defined by $f=0$. We have $v_i(f)
\geqslant v_i(\lceil D \rceil) \geqslant v_i(D)$ for all $E_i$.
Consider any exceptional divisor $E_i$ where $D$ has negative
excess, from the inequality $(\overline{C}_{\rm exc}-v_i(f)E_i)\cdot
E_i \geqslant (\lceil D \rceil-v_i(\lceil D \rceil) E_i)\cdot E_i$
we deduce
$$-v_i(f)E_i \cdot E_i \geqslant (\lceil D \rceil-v_i(\lceil D \rceil) E_i)\cdot E_i $$
just because $\overline{C}_{\rm exc}\cdot E_i \leqslant 0$.
Equivalently $(\lceil D \rceil+(v_i(f)-v_i(\lceil D \rceil))E_i)
\cdot E_i \leqslant 0 $ so it follows that $n_i \leqslant
v_i(f)-v_i(\lceil D \rceil)$. In particular $n_i +v_i(\lceil D
\rceil) \leqslant v_i(f)$ and $f\in I_{D'}$.

\end{proof}

\begin{proposition}
The antinef closure $\widetilde{D}$ of a divisor $D\in {\rm
Div}_{\bQ}(X')$ is achieved after finitely many unloading steps.
\end{proposition}

\begin{proof}
We want to show that the divisors in the sequence
$$D\leqslant D_1=\lceil D \rceil <\cdots < D_t < D_{t+1} < \cdots$$
obtained during the unloading procedure are all contained in the
antinef closure $\widetilde{D}$, then the result will follow since
both $D_1$ and $\widetilde{D}$ have integral coefficients and the
inequalities in the unloading sequence are strict. Clearly
$D_1\leqslant \widetilde{D}$ and suppose that $D_t\leqslant
\widetilde{D}$. Notice that for any component $E_i \leqslant D_t$
with negative excess we have
 $(\widetilde{D} - D_t) \cdot E_i \leqslant - D_t \cdot E_i $.
 Then, if we denote $\widetilde{D} - D_t= \sum_{i} m_iE_i$,
the previous inequality becomes $$\begin{array}{rclcl}
(\widetilde{D}-D_t)\cdot E_i&=& (m_iE_i+\sum_{j\neq i} m_jE_j)\cdot
E_i\\&=&m_iE_i^2+\sum_{j\neq i} m_jE_j\cdot E_i&\leqslant&- D_t
\cdot E_i.
\end{array}\,$$
Then, using that $\sum_{j\neq i} m_jE_j\cdot E_i\geqslant 0$, we get
$$m_i\geqslant \left\lceil\frac{- D_t \cdot E_i}{E_i^2}\right\rceil\,,$$
where we used the fact that $D_t$ and $\widetilde{D}$ have integer
coefficients. It follows that $D_{t+1}$ is also contained in
$\widetilde{D}$.
\end{proof}

\subsection{Multiplier ideals}\label{mult_ideals}

Let $\pi: X' \rightarrow X$ be a log-resolution of an ideal $\fa
\subseteq \cO_{X}$ and let  $F$ be the divisor such that
$\fa\cdot\cO_{X'} = \cO_{X'}\left(-F\right)$. The {\em multiplier
ideal (sheaf)} associated to $\fa $ and some rational number
$\lambda \in \bQ_{>0}$ is defined as\footnote{By an abuse of
notation, we will also denote $\J\left(\fa^\lambda\right)$ its stalk
at $O$ so we will omit the word ''sheaf'' if no confusion arises.}
$$\J\left(\fa^\lambda\right) = \pi_*\cO_{X'}\left(\left\lceil K_{\pi} - \lambda F \right\rceil\right).$$
For a detailed overview of the theory of multiplier ideals and the
properties they satisfy, we must refer to the book of Lazarsfeld
\cite{Laz04}. For more details in the case that $X$ has rational
singularities we also recommend to take a look at \cite{Tuc09} and
\cite{Tuc10}.

\vskip 2mm

The definition of multiplier ideals is independent of the choice of
log resolution. For simplicity, we will always fix a given
resolution. Multiplier ideals are complete and they are invariants
up to integral closure, i.e.
$\J(\fa^\lambda)=\J(\overline{\fa}^\lambda)$, therefore, without
loss of generality, we may always assume that the ideal $\A$ is
complete.  Moreover,  if $\fa $ is $\fM$-primary it follows that its
associated multiplier ideals $\J\left(\fa^{\lambda}\right)$ are
$\fM$-primary as well.

\vskip 2mm

Some other important properties of multiplier ideals that we will
use in this work are:

\begin{itemize}
 \item {\it Local vanishing theorem:}  $R^i \pi_*\cO_{X'}\left(\left\lceil K_{\pi} - \lambda F \right\rceil\right)= 0$ for all $i>0$ and all $\lambda\in \bQ_{>0}$.

 \item {\it Skoda's theorem:} $\J\left(\fa^\lambda\right) = \fa \cdot \J\left(\fa^{\lambda-1}\right) $ for all $\lambda > \dim \cO_{X,O}=2$.

 %\item {\it Skoda's theorem for principal ideals:}  $\J\left(\fa^\lambda\right) = \fa \cdot \J\left(\fa^{\lambda-1}\right) $ for all $\lambda \geqslant 1$.
\end{itemize}

For the case of principal ideals  there is another version of
Skoda's theorem that states that $\J\left(\fa^\lambda\right) = \fa
\cdot \J\left(\fa^{\lambda-1}\right) $ for all $\lambda \geqslant
1$. In particular, we have peridiocity of jumping numbers.

\vskip 2mm

Multiplier ideals come with an attached set of invariants that were
studied systematically by Ein-Lazarsfeld-Smith-Varolin in
\cite{ELSV04}. Clearly
$$\left\lceil K_{\pi}- \lambda F \right\rceil \geqslant \left\lceil K_{\pi}-\left(\lambda +\varepsilon\right)F\right\rceil$$
for any $\varepsilon > 0$, with equality if $\varepsilon$ is small
enough. Therefore the multiplier ideals form a discrete nested
sequence of ideals
$$\Oc_{X,O}\supseteq\J(\A^{\lambda_0})\varsupsetneq\J(\A^{\lambda_1})\varsupsetneq \J(\A^{\lambda_2})\varsupsetneq...\varsupsetneq\J(\A^{\lambda_i})\varsupsetneq...$$
indexed by an increasing sequence of rational numbers $0=\lambda_0 <
\lambda_1 < \lambda_2 < \ldots$ such that for any $c \in
[\lambda_i,\lambda_{i+1})$ it holds
$$\J(\A^{\lambda_i})=\J(\A^c)\varsupsetneq\J(\A^{\lambda_{i+1}}).$$
The $\lambda_i$ are the so-called {\em jumping numbers} of the ideal
$\fa$ and the first jumping number $\lambda_1={\rm lct}(\fa)$ is the
\emph{log-canonical threshold} of $\A$.

\subsection{Contributing divisors}\label{Tucker}

The jumps between multiplier ideals necessarily must occur at
rational numbers $\lambda\in \bQ$ which cause the strict inclusion
of divisors $$\left\lceil K_{\pi}- \lambda F \right\rceil <
\left\lceil K_{\pi}-\left(\lambda-\varepsilon\right)F\right\rceil $$
for any $\varepsilon$. If we take a close look at
 $F= F_{\rm exc}+ F_{\rm aff}$  these rational numbers must belong to
the set of {\it candidate jumping numbers}
$$\left \{  \frac{k_i + m}{e_i} \hskip 2mm | \hskip 2mm m\in \bZ_{>0}  \right\}$$
Notice that for non-exceptional components $E_i \leqslant F_{\rm
aff}$ we have $k_i=0$ and their corresponding candidates $\left \{
\frac{m}{e_i} \hskip 2mm | \hskip 2mm m\in \bZ_{>0}  \right\}$ are
indeed jumping numbers.

\vskip 2mm

It is easy to check that not every candidate jumping number (coming
from the exceptional part) is necessarily a jumping number. To
separate the wheat from the chaff, Tucker \cite{Tuc10} developed the
notion of {\it divisor that contributes} to a jumping number,
building upon previous work by Smith-Thompson \cite{ST07}.

\begin{definition}
 A positive rational number $\lambda$ is a \emph{candidate jumping number} for a reduced divisor $G\leqslant F$ if it
satisfies $\lambda e_i - k_i \in \bZ_{>0}$ for any component
$E_i\leqslant G$.

\end{definition}

\begin{definition}\cite[Def. 3.1]{Tuc10}
A reduced divisor $G \leqslant F$ for which $\lambda$ is a candidate
jumping number is said to {\it contribute} to $\lambda$ if
$$ \pi_{*}\Oc_{X'}(\lceil K_{\pi} -\lambda F\rceil+G)\varsupsetneq \J({\A}^{\lambda})$$
Moreover, this contribution is {\it critical} if
 for any divisor $0 \leqslant G'< G  $ we have
$$ \pi_{*}\Oc_{X'}(\lceil K_{\pi} -\lambda F\rceil+G') = \J({\A}^{\lambda}).$$
\end{definition}

Most often we will simply say that $G$ is just a {\it contributing}
or a {\it critical divisor} associated to $\lambda$. Critical
divisors define complete ideals very close to a multiplier ideal in
a precise sense that will be explained in the forthcoming Corollary
\ref{critical_min} in Section \S\ref{mini}. One may identify critical
divisors with exceptional support through the following numerical
characterization.

\begin{proposition} \label{T1}\cite[Thm. 4.3]{Tuc10}
Let $\lambda$ be a candidate jumping number for a reduced divisor
$G\in \Lambda$ with connected support.

\begin{itemize}
 \item[$\cdot$]  If $G=E_i$ is prime, then $E_i$ is a critical divisor for $\lambda$  if and only if
 $$(\lceil K_{\pi} - \lambda F\rceil + E_i)\cdot E_i \geqslant 0.$$

 \item[$\cdot$] If $G$ is reducible, then  $G$ is a critical divisor for $\lambda$ if and only if
 $$(\lceil K_{\pi} - \lambda F\rceil + G )\cdot E_i = 0 $$ for all divisors $E_i$ in the support of $G$.
\end{itemize}
\end{proposition}

Moreover, critical divisors with exceptional support satisfy a nice
geometric property when viewed in the dual graph.

\begin{proposition} \label{T2}\cite[Cor. 4.2 \& Thm 5.1]{Tuc10}
Let $G$ be a critical divisor for a jumping number $\lambda$. Then
$G$ is a connected chain in the dual graph of the log-resolution of
$\A$ whose ends must be either rupture or dicritical divisors.
\end{proposition}

Using all these properties, Tucker provides a simple  algorithm to
compute the set of all jumping numbers (see \cite[\S 6]{Tuc10}). It
boils down to the following steps:

\begin{algorithm} { (Jumping Numbers)} \label{A1}

\vskip 2mm

\noindent {\tt Input:} { A log-resolution of an 
 ideal} $\fa \subseteq \cO_{X,O}$.  %\newline

\noindent {\tt Output:} List of Jumping Numbers of $\fa$.%\newline

 \vskip 2mm

 \begin{enumerate}
  \item[$\bullet$] {\bf Jumping number}:

  \begin{itemize}
  \item[$\cdot$]  Compute the candidate jumping numbers for $F_{\rm exc}$.
  \item[$\cdot$]  Find all possible critical divisors using Prop. \ref{T2}.
  \item[$\cdot$]  Find which candidate jumping numbers can be realized as jumping number associated to these critical divisors
  using Prop. \ref{T1}.
  \item[$\cdot$] Plug in those jumping numbers coming from $F_{\rm aff}$.

  \end{itemize}

 \end{enumerate}

\end{algorithm}

\section{An algorithm to compute jumping numbers and multiplier ideals} \label{main_algorithm}

The aim of this section is to compute the jumping numbers and their
corresponding multiplier ideals of any given ideal  $\fa \subseteq
\cO_{X,O}$. To such purpose, we fix a log-resolution $\pi: X' \lra
X$ of our ideal $\fa$. The main ingredients we will have to deal
with are the relative canonical divisor $K_\pi = \sum_{i=1}^r k_i
E_i \in \Lambda_{\bQ}$, and the divisor $F \in {\rm Div}(X')$ such
that $\fa \cO_{X'}= \cO_{X'}(-F)$. Recall that we have a
decomposition
$$ F=F_{\rm exc}+F_{\rm aff}= \sum_{i=1}^r e_i E_i + \sum_{i=r+1}^s e_i E_i $$ in terms of its exceptional and affine support.

\vskip 2mm

We will provide a very simple algorithm that allows one to construct
sequentially the chain of multiplier ideals\footnote{In fact, we can
compute the chain inside any desired fixed range $[c, c']\subseteq
\mathbb{R}$:
$$ \J(\A^{c})= \J(\A^{\lambda_0})\varsupsetneq\J(\A^{\lambda_1})\varsupsetneq...\varsupsetneq\J(\A^{\lambda_r})= \J(\A^{c'})\, . $$ }
$$\Oc_{X,O}\supseteq\J(\A^{\lambda_0})\varsupsetneq\J(\A^{\lambda_1})\varsupsetneq \J(\A^{\lambda_2})\varsupsetneq...\varsupsetneq\J(\A^{\lambda_i})\varsupsetneq...$$

When $X$ is a smooth surface, or even when $X$ has a log-terminal
singularity at $O$, the multiplier ideal associated to $\lambda_0=0$
is the whole ring, i.e. $\Oc_{X,O}= \J(\A^{\lambda_0})$. In general,
when $X$ has a rational singularity  we may have an strict inclusion
$\Oc_{X,O}\varsupsetneq\J(\A^{\lambda_0})$. The starting point of
our method will be describing this ideal by means of the antinef
closure $D_{\lambda_0}=\sum e_i^{\lambda_0} E_i$ of $\lfloor  -
K_\pi \rfloor$ that we compute using the unloading procedure
described in Section \S\ref{unloading}.

\vskip 2mm

As a consequence of our main result (see Theorem \ref{lct_ideal}),
the log-canonical threshold satisfies the following
formula\footnote{When $X$ is smooth, or even when it has
log-terminal singularities, we have $D_{\lambda_0}=0$ so one
recovers the well-known formula for the log-canonical threshold.}
\begin{equation} \label{lct-min}
\lambda_1= {\rm lct}({\A})= \min_i\left\{\frac{k_i+1 +
e_i^{\lambda_0}}{e_i}\right\}.
\end{equation}

\vskip 2mm

Then we describe its associated multiplier ideal
$\J(\A^{\lambda_1})$  just computing the antinef closure
$D_{\lambda_1}$ of $\lfloor \lambda_1 F - K_\pi \rfloor$ using the
unloading procedure. Once we have the divisor $D_{\lambda_1}$, we
use an extension of Formula \ref{lct-min}  given by Theorem
\ref{lct_ideal},  that computes the next jumping number $\lambda_2$.
Then we only have to follow the same strategy: the antinef closure
$D_{\lambda_2}$ of $\lfloor \lambda_2 F - K_\pi \rfloor$, i.e. the
multiplier ideal $\J(\A^{\lambda_2})$, will allow us to compute
$\lambda_3$ and so on.

\vskip 2mm

The main idea behind our method is a simple comparison between
complete ideals. Whenever we have two antinef divisors it is easy to
check  whether their corresponding complete  ideals satisfy a strict
inclusion (see Proposition \ref{semicont_antinef}). To compare the
ideals associated to an antinef and a non-antinef divisor is more
subtle and this is the situation that we will have to deal with in
this section.

\vskip 2mm

To address this problem we will need some preliminary technical
results.

\begin{lemma}\label{Lem1}
Let $D_1,D_2$ be two divisors in ${\rm Div}(X')$ such that
$D_1\leqslant D_2$. Then, they have the same antinef closure
$\widetilde{D_1}=\widetilde{D_2}$ if and only if
$\widetilde{D_1}\geqslant D_2$.

\end{lemma}

\begin{proof}
Recall that, by Proposition \ref{prop_clausura}, we already have
$\widetilde{D_1}\leqslant\widetilde{D_2}$ just because $D_1\leqslant
D_2$.

\vskip 2mm

Assume $\widetilde{D_1}\geqslant D_2$ then, by the definition of
antinef closure (see Lemma \ref{def_antinef}), we also have
$\widetilde{D_1}\geqslant \widetilde{D_2}\geqslant  D_2$ and thus
$\widetilde{D_1}=\widetilde{D_2}$. On the other hand, assume that
$\widetilde{D_1}=\widetilde{D_2}$. Then, since the antinef closure
of a divisor always contains it, we have
$\widetilde{D_1}=\widetilde{D_2} \geqslant D_2$ as desired.
\end{proof}

\begin{corollary}
Let $D_1,D_2$ be two divisors in ${\rm Div}(X')$ such that
$D_1\leqslant D_2$. Then, $\widetilde{D_1}<\widetilde{D_2}$ if and
only if $v_i(\widetilde{D_1}) < v_i(D_2)$ for some $E_i$.
\end{corollary}

\begin{proof}
As  $D_1\leqslant D_2$, the inclusion
$\widetilde{D_1}\leqslant\widetilde{D_2}$ also holds. The result
then follows from Lemma \ref{Lem1}.
\end{proof}

Translated into the language of complete ideals, these results give
a characterization of the jump between two nested ideals, which will
be a key ingredient in the proof of our results.

\begin{proposition}\label{semicont}
Let $D_1,D_2$ be two divisors in ${\rm Div}(X')$ such that
$D_1\leqslant D_2$. Then:

\begin{itemize}

\item[i)] $\pi_* \Oc_{X'}(-D_1)= \pi_* \Oc_{X'}(-D_2)$ if and only
if $\widetilde{D_1}\geqslant D_2$.

\item[ii)] $\pi_* \Oc_{X'}(-D_1) \varsupsetneq \pi_* \Oc_{X'}(-D_2)$ if and only if
$v_i(\widetilde{D_1}) < v_i(D_2)$ for some $E_i$.
\end{itemize}

\end{proposition}

For convenience we also present this result in the form we will most
commonly use it.

\begin{corollary} \label{semicontMI}
 Let $\lambda' <\lambda$ be rational numbers.
Let $D_{\lambda'}=\sum e^{\lambda'}_i E_i$ be the antinef closure of
$\left\lfloor \lambda' F  - K_{\pi} \right\rfloor$. Then:

\begin{itemize}

\item[i)] $\J(\A^{\lambda'}) = \J(\A^{\lambda}) $
if and only if $\lfloor \lambda e_i - k_i \rfloor \leqslant
e^{\lambda'}_i$ for all $E_i$.

\item[ii)] $\J(\A^{\lambda'}) \varsupsetneq
\J(\A^{\lambda})$ if and only if $\lfloor \lambda e_i - k_i \rfloor
> e^{\lambda'}_i$ for some $E_i$.
\end{itemize}

\end{corollary}

\vskip 2mm

With the technical tools stated above we are ready for the main
result of this section.

\begin{theorem} \label{lct_ideal}
Let $\fa \subseteq \cO_{X,O}$ be an ideal and let
$D_{{\lambda'}}=\sum e_i^{{\lambda'}} E_i$ be the antinef closure of
$\left\lfloor \lambda' F  - K_{\pi} \right\rfloor$ for a given $
{\lambda}'\in \bQ_{>0}$. Then, $$\lambda= \min_i\left\{\frac{k_i+1+
e_i^{\lambda'}}{e_i}\right\}$$ is the jumping number consecutive to
$\lambda'$.

%Moreover $\lambda$ is the log-canonical threshold of $\fa$ with respect to the previous multiplier ideal $\J(\fa^{\lambda '})$.
\end{theorem}

\begin{proof}
Let us check first that $\lambda'<\lambda$. Indeed, by the
definition of antinef closure, the integers $e_i^{\lambda'}$ satisfy
$\lfloor \lambda' e_i - k_i \rfloor \leqslant e^{\lambda'}_i$ for
any $E_i$, and hence:
$$\lambda' <\frac{k_i+1+e_i^{\lambda'}}{e_i}\,.$$
Thus, we have an inclusion of ideals $\J(\A^{\lambda'})\supseteq
\J(\A^{\lambda}).$ Notice that for those divisors $E_i$ where the
minimum is achieved we have
$$\lfloor \lambda e_i - k_i \rfloor = 1 + e^{\lambda'}_i > e^{\lambda'}_i$$ so the above inclusion of ideals is strict by Corollary
 \ref{semicontMI}. To conclude that $\lambda$ is the jumping number immediately after $\lambda'$, we have to show that for any
$c\in\R$ with $\lambda'\leqslant c<\lambda$ we have
$\J(\A^{\lambda'})= \J(\A^{c})$. Suppose the contrary, i.e.,
$\J(\A^{\lambda'}) \varsupsetneq \J(\A^{c})$. By Corollary
\ref{semicontMI}, this $c$ should satisfy $\lfloor \lambda e_i - k_i
\rfloor
> e^{\lambda'}_i$ or equivalently
$c\geqslant\frac{k_i+1+e_i^{\lambda'}}{e_i}$ for some $E_i$, and
this contradicts the fact that $\lambda$ is the minimum of these
rational numbers.
\end{proof}

The above result for the case  $\lambda'=0$ gives a mild
generalization of the well-known formula for the log-canonical
threshold in the smooth case. We point out that the antinef closure
of $\lfloor -K_\pi \rfloor$ is $0$ whenever $X$ is smooth or, more
generally, when it has log-terminal singularities.

\begin{corollary}
 Let $\fa \subseteq \cO_{X,O}$ be an  ideal. Let
$D_{\lambda_0}=\sum e_i^{\lambda_0} E_i$ be the antinef closure of
$\left\lfloor - K_{\pi} \right\rfloor$. Then, $${\rm lct}(\fa)=
\min_i\left\{\frac{k_i+1+ e_i^{\lambda_0}}{e_i}\right\}.$$

\end{corollary}

Another easy application of the results above is the following
result that should be well-known to experts.

\begin{corollary}
Let $  {\lambda}_1$ be the log-canonical threshold of an ideal $\fa
\subseteq \cO_{X,O}$ and assume that $X$ has at most a log-terminal
singularity at $O$. Then $\J(\fa^{\lambda_1})=\fM$.
\end{corollary}

\begin{proof}
Since $X$ has at most a log-terminal singularity, the log-canonical
threshold is $${\rm lct}(\fa)= \lambda_1=
\min_i\left\{\frac{k_i+1}{e_i}\right\}$$ so it satisfies $\lambda_1
\leqslant \frac{k_i+1}{e_i}$ for any divisor $E_i$ and equality is
achieved at least for a given divisor. In particular,  for all $E_i$
we have
$$ \lfloor \lambda_1 e_i -k_i \rfloor \leqslant 1.$$ It follows from Proposition \ref{prop_clausura}
 that $\fM \subseteq \J(\fa^{\lambda_1}) \varsubsetneq \cO_{X,O}$ and we get the desired result.
\end{proof}

\vskip 2mm

For non log-terminal singularities we may find examples where the
codimension as $\bC$-vector spaces of
$\J(\fa^{\lambda_0})\varsupsetneq \J(\fa^{\lambda_1})$ might be
bigger than $1$ (see Example \ref{ex2}).

\vskip 2mm

Combining Theorem \ref{lct_ideal} and the unloading procedure
described in Section \S\ref{unloading}  we can describe a very
simple algorithm that allows us to compute the chain of multiplier
ideals:

\vskip 2mm

\begin{algorithm} { (Jumping Numbers and  Multiplier Ideals)} \label{A2}

\vskip 2mm

\noindent {\tt Input:} { A log-resolution of an 
 ideal} $\fa \subseteq \cO_{X,O}$. %\newline

\noindent {\tt Output:} List of Jumping Numbers of $\fa$ and its
corresponding
Multiplier Ideals. %\newline

 \vskip 2mm

Set $\lambda_0=0$ and compute the antinef closure
$D_{\lambda_0}=\sum e_i^{\lambda_0}E_i$ of $\lfloor - K_{\pi}
\rfloor$ using the unloading
  procedure.   From $j=1$ , incrementing by $1$

\vskip 2mm

\begin{itemize}

\item[\textbf{(Step j)}]

 \begin{enumerate}
  \item[$\cdot$] {\bf Jumping number}: Compute
     $$\lambda_j=\min_i\left\{\frac{k_i+1+e_i^{\lambda_{j-1}}}{e_i}\right\}.$$

       \vskip 2mm

  \item [$\cdot$] {\bf Multiplier ideal}: Compute the antinef closure $D_{\lambda_j}=\sum e_i^{\lambda_j}E_i$ of $\lfloor \lambda_j F - K_{\pi} \rfloor$ using the unloading
  procedure.

 % \vskip 2mm

  %\item [$\cdot$] {\bf Multiplicity of the jumping number}: Compute the multiplicity of $\lambda$ as
  % $$m(\lambda)=  \frac{D_{\lambda_{j-1}} \cdot(D_{\lambda_{j-1}}+K_\pi)}{2} -\frac{D_{\lambda_j} \cdot(D_{\lambda_j}+K_\pi)}{2}.$$
 \end{enumerate}

 \end{itemize}

\end{algorithm}

\vskip 2mm

Notice that we may also find all the multiplier ideals in any given
interval $[c',c]$ of the real line. In this case, our starting point
would be computing the antinef closure $D_{c'}$ of $\lfloor c' F -
K_{\pi} \rfloor$. To illustrate this method we consider an easy
example in a smooth variety.

%\newpage
\begin{example}\label{ex1}
Consider the ideal $\fa=(x^2y^2,x^5,y^5,xy^4,x^4y) \subseteq
\bC\{x,y\}$. We represent the relative canonical divisor $K_\pi$ and
the divisor $F$ in the dual graph as follows:

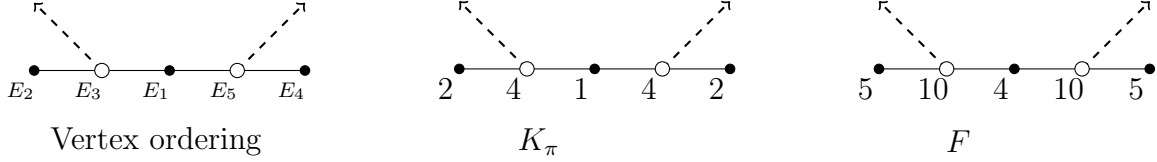
\setlength{\unitlength}{7.5mm}

\begin{center}
 \begin{tabular}{ccc}

   \begin{tikzpicture}[scale=0.9]
   \draw  (-2,0) -- (2,0);
   \draw [dashed,->,thick] (1,0) -- (2,1);
   \draw [dashed,->,thick] (-1,0) -- (-2,1);
   \draw (-.2,-0.3) node {{\tiny $E_1$}};
   \draw (-2.2,-0.3) node {{\tiny $E_2$}};
   \draw (-1.2,-0.3) node {{\tiny $E_3$}};
   \draw (1.8,-0.3) node {{\tiny $E_4$}};
   \draw (0.8,-0.3) node {{\tiny $E_5$}};
   \filldraw  (0,0) circle (2pt)
              (2,0) circle (2pt)
              (-2,0) circle (2pt);
   \filldraw  [fill=white]  (-1,0) circle (3pt)
              (1,0) circle (3pt);
 \end{tikzpicture} &
\hspace{10mm} \begin{tikzpicture}[scale=0.9]
   \draw  (-2,0) -- (2,0);
   \draw [dashed,->,thick] (1,0) -- (2,1);
   \draw [dashed,->,thick] (-1,0) -- (-2,1);
   \draw (-.2,-0.3) node {$1$};
   \draw (-2.2,-0.3) node {$2$};
   \draw (-1.2,-0.3) node {$4$};
   \draw (1.8,-0.3) node {$2$};
   \draw (0.8,-0.3) node {$4$};
   \filldraw  (0,0) circle (2pt)
              (2,0) circle (2pt)
              (-2,0) circle (2pt);
   \filldraw  [fill=white]  (-1,0) circle (3pt)
              (1,0) circle (3pt);
 \end{tikzpicture}
& \hspace{10mm}  \begin{tikzpicture}[scale=0.9]
   \draw  (-2,0) -- (2,0);
   \draw [dashed,->,thick] (1,0) -- (2,1);
   \draw [dashed,->,thick] (-1,0) -- (-2,1);
   \draw (-.2,-0.3) node {$4$};
   \draw (-2.2,-0.3) node {$5$};
   \draw (-1.2,-0.3) node {$10$};
   \draw (1.8,-0.3) node {$5$};
   \draw (0.8,-0.3) node {$10$};
   \filldraw  (0,0) circle (2pt)
              (2,0) circle (2pt)
              (-2,0) circle (2pt);
   \filldraw  [fill=white]  (-1,0) circle (3pt)
              (1,0) circle (3pt);
 \end{tikzpicture}\\
\parbox{30mm}{\begin{center} Vertex ordering\end{center}}&
\parbox{30mm}{\begin{center} $K_\pi$\end{center}}&
\parbox{30mm}{\begin{center} $F$\end{center}}

 \end{tabular}

\end{center}

The blank dots correspond to dicritical divisors and their excesses
are represented by broken arrows\footnote{The broken arrows also
represent the branches of the strict transform of a curve defined by
a generic $f\in \fa$.}. For simplicity we will collect the values of
any divisor in a vector. To start with we have $K_\pi=(1,2,4,2,4)$
and $F=(4,5,10,5,10)$. In the algorithm we will have to perform some
unloading steps so we will have to consider the intersection matrix
$M=(E_i\cdot E_j)_{1\leqslant i,j \leqslant 5}$
$$M=\left(\begin{array}{rrrrr}
      -5&    0&  1&  0&  1\\
       0&   -2&  1&  0&  0\\
       1&    1& -1&  0&  0\\
       0&    0&  0& -2&  1\\
       1&    0&  0&  1& -1\\
        \end{array}
\right).
$$

\vskip 5mm

The algorithm is performed as follows:

\vskip 2mm

$\bullet$ We start computing the log-canonical threshold:
$$\lambda_1= {\rm lct}(\fa)=\min_i\left\{\frac{k_i+1}{e_i}\right\}= \min_i\left\{\frac{2}{4},\frac{3}{5},\frac{5}{10},\frac{3}{5},\frac{5}{10}\right\}=\frac{1}{2}.$$
The divisor $\lfloor\frac{1}{2}F - K_\pi\rfloor=(1,0,1,0,1)$ is not
antinef { since it has excess $-1$ at $E_2$ and $E_4$. The first unloading step is to consider the divisor
$\lfloor\frac{1}{2}F - K_\pi\rfloor +E_2+E_4=(1,1,1,1,1)$. This divisor has excess $-1$ at $E_3$ and $E_5$ so we need to perform a 
second unloading step to obtain the antinef closure} $D_{\lambda_1}=(1,1,2,1,2)$.

\vskip 2mm

$\bullet$ The second Jumping Number is:
$$\lambda_2=\min_i\left\{\frac{k_i+1 + e_i^{\lambda_{1}}}{e_i}\right\}=
\min_i\left\{\frac{2+1}{4},\frac{3+1}{5},\frac{5+2}{10},\frac{3+1}{5},\frac{5+2}{10}\right\}=\frac{7}{10}.$$
Then we get $\lfloor\frac{7}{10}F - K_\pi\rfloor=(1,1,3,1,3)$. { It has excess $-1$ at $E_1,E_2$ and $E_4$ and 
we obtain the divisor $(2,2,3,2,3)$ after the first unloading step. This divisor has excess $-1$ at $E_3$ and $E_5$ 
and, after a second unloading step, we obtain the antinef closure} $D_{\lambda_2}=(2,2,4,2,4)$.

\vskip 2mm

$\bullet$ The third Jumping Number is:
$$\lambda_3=\min_i\left\{\frac{k_i+1 + e_i^{\lambda_{2}}}{e_i}\right\}=
\min_i\left\{\frac{2+2}{4},\frac{3+2}{5},\frac{5+4}{10},\frac{3+2}{5},\frac{5+4}{10}\right\}=\frac{9}{10}.$$
Then we get $\lfloor\frac{9}{10}F - K_\pi\rfloor=(2,2,5,2,5)$ { that has excess $-1$ at $E_3$ and $E_5$. After a
single unloading step we get the antinef closure} $D_{\lambda_3}=(2,3,5,3,5)$.

\vskip 2mm

$\bullet$ The fourth Jumping Number is:
$$\lambda_4=\min_i\left\{\frac{k_i+1 + e_i^{\lambda_{3}}}{e_i}\right\}=
\min_i\left\{\frac{2+2}{4},\frac{3+3}{5},\frac{5+5}{10},\frac{3+3}{5},\frac{5+5}{10}\right\}=
1.$$ Then we get $\lfloor F - K_\pi\rfloor= D_{\lambda_4}
=(3,3,6,3,6)$ since this divisor is antinef.

\vskip 2mm

$\bullet$ The fifth Jumping Number is:
$$\lambda_5=\min_i\left\{\frac{k_i+1 + e_i^{\lambda_{4}}}{e_i}\right\}=
\min_i\left\{\frac{2+3}{4},\frac{3+3}{5},\frac{5+6}{10},\frac{3+3}{5},\frac{5+6}{10}\right\}=\frac{11}{10}.$$
Then we get $\lfloor\frac{11}{10}F - K_\pi\rfloor=(3,3,7,3,7)$ and, { after a single unloading step, we obtain the
antinef closure} $D_{\lambda_5}=(3,4,7,4,7)$.

\vskip 3mm

Now we will compute the chain of multiplier ideals of the  plane
curve defined by $f=(x^2-y^3)(y^2-x^3)\in \bC\{x,y\}$. The product of two
cusps sharing the origin $O$ is a generic element of the ideal
$\fa=(x^2y^2,x^5,y^5,xy^4,x^4y)$ considered above, so
$\J(f^\lambda)=\J(\A^\lambda)$ for $\lambda < 1$. This example will
illustrate how the non-exceptional components affect the unloading
procedure and, consequently, the list of jumping numbers for
$\lambda > 1$.

\vskip 3mm

Denote the total transform of the curve defined by  $f$ simply as
$F$. We represent the relative canonical divisor $K_\pi$ and the
divisor $F$ in the dual graph as follows:

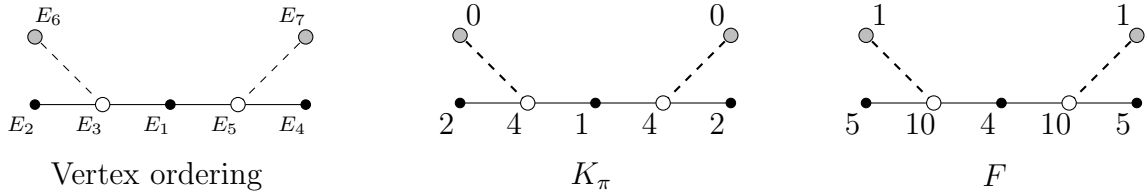
\setlength{\unitlength}{7.5mm}

\begin{center}
 \begin{tabular}{ccc}

    \begin{tikzpicture}[scale=0.9]
   \draw  (-2,0) -- (2,0);
   \draw [dashed] (1,0) -- (2,1);
   \draw [dashed] (-1,0) -- (-2,1);
   \draw (-.2,-0.3) node {{\tiny $E_1$}};
   \draw (-2.2,-0.3) node {{\tiny $E_2$}};
   \draw (-1.2,-0.3) node {{\tiny $E_3$}};
   \draw (1.8,-0.3) node {{\tiny $E_4$}};
   \draw (0.8,-0.3) node {{\tiny $E_5$}};
   \draw (-1.8,1.3) node {{\tiny $E_6$}};
   \draw (1.8,1.3) node {{\tiny $E_7$}};
   \filldraw  (0,0) circle (2pt)
              (2,0) circle (2pt)
              (-2,0) circle (2pt);
   \filldraw  [fill=white]  (-1,0) circle (3pt)
              (1,0) circle (3pt);
   \filldraw  [fill=grey]  (-2,1) circle (3pt)
              (2,1) circle (3pt);
 \end{tikzpicture} &
\hspace{10mm}  \begin{tikzpicture}[scale=0.9]
   \draw  (-2,0) -- (2,0);
   \draw [dashed,thick] (1,0) -- (2,1);
   \draw [dashed,thick] (-1,0) -- (-2,1);
   \draw (-.2,-0.3) node {$1$};
   \draw (-2.2,-0.3) node {$2$};
   \draw (-1.2,-0.3) node {$4$};
   \draw (1.8,-0.3) node {$2$};
   \draw (0.8,-0.3) node {$4$};
   \draw (-1.8,1.3) node {{$0$}};
   \draw (1.8,1.3) node {{$0$}};
   \filldraw  (0,0) circle (2pt)
              (2,0) circle (2pt)
              (-2,0) circle (2pt);
   \filldraw  [fill=white]  (-1,0) circle (3pt)
              (1,0) circle (3pt);
   \filldraw  [fill=grey]  (-2,1) circle (3pt)
              (2,1) circle (3pt);
 \end{tikzpicture}
\hspace{10mm}   \begin{tikzpicture}[scale=0.9]
   \draw  (-2,0) -- (2,0);
   \draw [dashed,thick] (1,0) -- (2,1);
   \draw [dashed,thick] (-1,0) -- (-2,1);
   \draw (-.2,-0.3) node {$4$};
   \draw (-2.2,-0.3) node {$5$};
   \draw (-1.2,-0.3) node {$10$};
   \draw (1.8,-0.3) node {$5$};
   \draw (0.8,-0.3) node {$10$};
   \draw (-1.8,1.3) node {{$1$}};
   \draw (1.8,1.3) node {{$1$}};
   \filldraw  (0,0) circle (2pt)
              (2,0) circle (2pt)
              (-2,0) circle (2pt);
   \filldraw  [fill=white]  (-1,0) circle (3pt)
              (1,0) circle (3pt);
   \filldraw  [fill=grey]  (-2,1) circle (3pt)
              (2,1) circle (3pt);
 \end{tikzpicture}\\
\parbox{30mm}{\begin{center} Vertex ordering\end{center}}&
\parbox{20mm}{\hskip -1.3cm  $K_\pi$ }&
\parbox{20mm}{ \hskip -2.5cm $F$ }

 \end{tabular}

\end{center}

\vskip 5mm

The gray dots will represent here the affine components
belonging to the strict transform of the curve. The intersection
matrix is now
$$M=\left(\begin{array}{rrrrrcc}
      -5&    0&  1&  0&  1& 0&  0\\
       0&   -2&  1&  0&  0& 0&  0\\
       1&    1& -1&  0&  0& 1&  0\\
       0&    0&  0& -2&  1& 0&  0\\
       1&    0&  0&  1& -1& 0&  1\\
        \end{array}
\right).
$$

\vskip 8mm

The algorithm is performed as follows:

\vskip 2mm

$\bullet$ The log-canonical threshold is:
$$\lambda_1= {\rm lct}(\fa)=\min_i\left\{\frac{k_i+1}{e_i}\right\}=
\min_i\left\{\frac{2}{4},\frac{3}{5},\frac{5}{10},\frac{3}{5},\frac{5}{10},
\frac{1}{1},\frac{1}{1}\right\}=\frac{1}{2}.$$ We get
$\lfloor\frac{1}{2}F - K_\pi\rfloor=(1,0,1,0,1,0,0)$ and, { as in the previous example, its antinef
closure is} $D_{\lambda_1}=(1,1,2,1,2,0,0)$.

\vskip 2mm

$\bullet$ The second Jumping Number is:
$$\lambda_2=\min_i\left\{\frac{k_i+1 + e_i^{\lambda_{1}}}{e_i}\right\}=
\min_i\left\{\frac{2+1}{4},\frac{3+1}{5},\frac{5+2}{10},\frac{3+1}{5},\frac{5+2}{10},\frac{1}{1},\frac{1}{1}\right\}=\frac{7}{10}.$$
Then we get $\lfloor\frac{7}{10}F - K_\pi\rfloor=(1,1,3,1,3,0,0)$
and its antinef closure $D_{\lambda_2}=(2,2,4,2,4,0,0)$.

\vskip 2mm

$\bullet$ The third Jumping Number is:
$$\lambda_3=\min_i\left\{\frac{k_i+1 + e_i^{\lambda_{2}}}{e_i}\right\}=
\min_i\left\{\frac{2+2}{4},\frac{3+2}{5},\frac{5+4}{10},\frac{3+2}{5},\frac{5+4}{10},\frac{1}{1},\frac{1}{1}\right\}=\frac{9}{10}.$$
Then we get $\lfloor\frac{9}{10}F - K_\pi\rfloor=(2,2,5,2,5,0,0)$
and its antinef closure $D_{\lambda_3}=(2,3,5,3,5,0,0)$.

\vskip 2mm

$\bullet$ The fourth Jumping Number is:
$$\lambda_4=\min_i\left\{\frac{k_i+1 + e_i^{\lambda_{3}}}{e_i}\right\}=
\min_i\left\{\frac{2+2}{4},\frac{3+3}{5},\frac{5+5}{10},\frac{3+3}{5},\frac{5+5}{10},\frac{1}{1},\frac{1}{1}\right\}=
1.$$ Then we get $\lfloor F - K_\pi\rfloor=(3,3,6,3,6,1,1)$ but this
divisor is not antinef because of the non-exceptional components.
{ Namely, we have excess $-1$ at $E_3$ and $E_5$.  To  obtain the antinef closure 
$D_{\lambda_4}=(4,5,10,5,10,1,1)$ we need to perform seven unloading steps with the intermediate divisors: 

\begin{itemize}
 \item[$\cdot$]  $(3,3,7,3,7,1,1)$ with excess $-1$ at $E_2$ and $E_4$.
  \item[$\cdot$]  $(3,4,7,4,7,1,1)$ with excess $-1$ at $E_3$ and $E_5$.
   \item[$\cdot$]  $(3,4,8,4,8,1,1)$ with excess $-1$ at $E_1$.
    \item[$\cdot$]  $(4,4,8,4,8,1,1)$ with excess $-1$ at $E_3$ and $E_5$.
     \item[$\cdot$]  $(4,4,9,4,9,1,1)$ with excess $-1$ at $E_2$ and $E_4$.
      \item[$\cdot$]  $(4,5,9,5,9,1,1)$ with excess $-1$ at $E_3$ and $E_5$.
   
\end{itemize}
 }

 \vskip 2mm 
 If we compare with the $\fM$-primary ideal $\fa$ we should notice that the
affine components of $\lfloor F - K_\pi \rfloor$ force us to add
more exceptional components when computing its antinef closure and
consequently, this will give a different jumping number in the next
step.

\vskip 2mm

$\bullet$ The fifth Jumping Number is:
$$\lambda_5=\min_i\left\{\frac{k_i+1 + e_i^{\lambda_{4}}}{e_i}\right\}=
\min_i\left\{\frac{2+4}{4},\frac{3+5}{5},\frac{5+10}{10},\frac{3+5}{5},\frac{5+10}{10},\frac{2}{1},\frac{2}{1}\right\}=\frac{3}{2}.$$
Then we get $\lfloor\frac{3}{2}F - K_\pi\rfloor=(5,5,11,5,11,1,1)$
and its antinef closure $D_{\lambda_5}=(5,6,12,6,12,1,1)$.

\end{example}

{

Consider a normal surface $X$ with a singularity at $O$. Given a minimal resolution $\pi:X' \lra X$  of $X$,
Artin \cite{Art66} introduced the {\it fundamental cycle} as the unique smallest non-zero effective divisor 
with exceptional support that is antinef. Moreover he proved that the singularity is rational if and only 
if the arithmetical genus of the  fundamental cycle is zero. 

\vskip 2mm 

We have that $\pi$ is also a minimal log-resolution 
of the maximal ideal $\fM \subseteq \cO_{X,O}$ and the fundamental cycle is the divisor $F$ such that 
$\fM\cdot\cO_{X'} = \cO_{X'}\left(-F\right)$. To compute its arithmetical genus we can use the formula
$ p_a(F)=1+ \frac{(K_\pi + F)\cdot F}{2}$ (see \cite{Art62}). 

\vskip 2mm 

This characterization gives us a good source of examples of surfaces with rational singularities.

 \begin{example} \label{ex2}

Consider a surface $X$  with a rational singularity at $O$ whose 
minimal resolution $\pi:X' \lra X$ has six exceptional components $E_1,\dots, E_6$
with the following dual graph and intersection matrix:

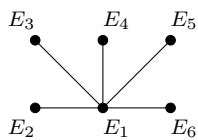
\begin{center}

\begin{tikzpicture}[scale=0.9]
   \draw  (-1,1) -- (0,0);
   \draw  (-1,0) -- (1,0);
   \draw  (1,1) -- (0,0);
   \draw  (0,1) -- (0,0);
   \draw (.2,-0.3) node {{\tiny $E_1$}};
   \draw (-1.2,-0.3) node {{\tiny $E_2$}};
   \draw (-1.2,1.3) node {{\tiny $E_3$}};
   \draw (.2,1.3) node {{\tiny $E_4$}};
   \draw (1.2,1.3) node {{\tiny $E_5$}};
   \draw (1.2,-0.3) node {{\tiny $E_6$}};
   \filldraw  (-1,1) circle (2pt)
              (-1,0) circle (2pt)
              (0,0) circle (2pt)
              (0,1) circle (2pt)
              (1,1) circle (2pt)
              (1,0) circle (2pt);
 \end{tikzpicture}

  \end{center}
 $$\left(\begin{array}{rrrrrr}
    -4&  1& 1&  1&  1&  1\\
    1&  -5& 0&  0&  0&  0\\
    1&   0& -5& 0&  0&  0\\
    1&   0& 0&  -5& 0&  0\\
    1&   0& 0&  0&  -5& 0\\
    1&   0& 0&  0&  0&  -5\\
        \end{array}
\right)$$

The  fundamental cycle is the divisor $F=(2,1,1,1,1,1)$ and the relative canonical divisor 
is $K_\pi= (-\frac{5}{3}, -\frac{14}{15}, -\frac{14}{15}, -\frac{14}{15},-\frac{14}{15}, -\frac{14}{15})$ so the
singularity is not even log-canonical.

 \vskip 2mm 
 
The multiplier ideals corresponding to $\lambda_0=0$ and
$\lambda_1={\rm lct}(\fM)=\frac{4}{9}$ are given by the antinef
divisors $D_{\lambda_0}=(2,1,1,1,1,1)$ and
$D_{\lambda_1}=(3,1,1,1,1,1)$. Notice that
 $\J(\fM^{\lambda_0})=\fM$ and, using the techniques of \cite{ACAMDCGA13}, we  get that
the codimension between these multiplier ideals is $4$.
\end{example}}

\subsection{Implementation}
We have implemented  Algorithm \ref{A2} in the Computer Algebra
system  {\tt Macaulay 2} \cite{GS}. The scripts of the source code
as well as the output in full detail of some examples are
available at the web page
\begin{center}
  {\tt  www.pagines.ma1.upc.edu/\textasciitilde{}jalvz/multiplier.html}
\end{center}
We implemented Tucker's Algorithm \ref{A1} as well in order to
compare both approaches. Of course, once we have the list of jumping
numbers we may use the unloading procedure of Section
\S\ref{unloading} to describe the corresponding multiplier ideals.
We have also implemented this extended version of Tucker's algorithm
and it turns out that our method is much faster.

\vskip 2mm For example, we have tested the case of an $\fM$-primary
ideal $\A$ whose corresponding dual graph has $35$ vertices
distributed in three branches only sharing the origin and each
branch has three rupture divisors.

 
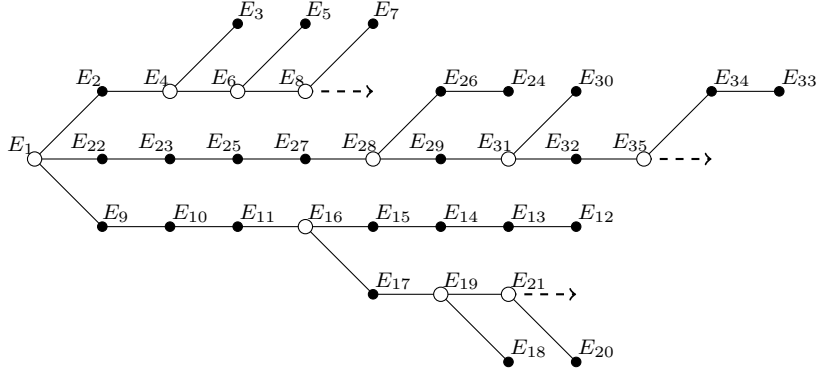
\begin{center}
  \begin{tikzpicture}[scale=0.9]
   \draw  (0,0) -- (1,1);
   \draw  (1,1) -- (4,1);
   \draw  (2,1) -- (3,2);
   \draw  (3,1) -- (4,2);
   \draw  (4,1) -- (5,2);
   \draw  (0,0) -- (9,0);
   \draw  (5,0) -- (6,1);
   \draw  (7,0) -- (8,1);
   \draw  (9,0) -- (10,1);
   \draw  (10,1) -- (11,1);
   \draw  (0,0) -- (1,-1);
   \draw  (1,-1) -- (8,-1);
   \draw  (7,-2) -- (8,-3);
   \draw  (6,1) -- (7,1);
   \draw  (4,-1) -- (5,-2);
   \draw  (5,-2) -- (7,-2);
   \draw  (6,-2) -- (7,-3);
   \draw [dashed,->,thick] (9,0) -- (10,0);
   \draw [dashed,->,thick] (4,1) -- (5,1);
   \draw [dashed,->,thick] (7,-2) edge (8,-2);
   \draw (3.2,2.2) node {\tiny $E_3$};
   \draw (4.2,2.2) node {\tiny $E_5$};
   \draw (5.2,2.2) node {\tiny $E_7$};
   \draw (0.8,1.2) node {\tiny $E_2$};
   \draw (0.8,0.25) node {\tiny $E_{22}$};
   \draw (1.8,0.25) node {\tiny $E_{23}$};
   \draw (2.8,0.25) node {\tiny $E_{25}$};
   \draw (3.8,0.25) node {\tiny $E_{27}$};
   \draw (5.8,0.25) node {\tiny $E_{29}$};
   \draw (7.8,0.25) node {\tiny $E_{32}$};
   \draw (1.2,-0.8) node {\tiny $E_{9}$};
   \draw (2.3,-0.8) node {\tiny $E_{10}$};
   \draw (3.3,-0.8) node {\tiny $E_{11}$};
   \draw (4.3,-0.8) node {\tiny $E_{16}$};
   \draw (5.3,-0.8) node {\tiny $E_{15}$};
   \draw (6.3,-0.8) node {\tiny $E_{14}$};
   \draw (7.3,-0.8) node {\tiny $E_{13}$};
   \draw (8.3,-0.8) node {\tiny $E_{12}$};
   \draw (5.3,-1.8) node {\tiny $E_{17}$};
   \draw (6.3,-1.8) node {\tiny $E_{19}$};
   \draw (7.3,-1.8) node {\tiny $E_{21}$};
   \draw (7.3,-2.8) node {\tiny $E_{18}$};
   \draw (8.3,-2.8) node {\tiny $E_{20}$};
   \draw (6.3,1.2) node {\tiny $E_{26}$};
   \draw (7.3,1.2) node {\tiny $E_{24}$};
   \draw (8.3,1.2) node {\tiny $E_{30}$};
   \draw (10.3,1.2) node {\tiny $E_{34}$};
   \draw (11.3,1.2) node {\tiny $E_{33}$};
   \draw (-.2,0.2) node {\tiny $E_{1}$};
   \draw (4.8,0.25) node {\tiny $E_{28}$};
   \draw (6.8,0.25) node {\tiny $E_{31}$};
   \draw (8.8,0.25) node {\tiny $E_{35}$};
   \draw (1.8,1.2) node {\tiny $E_{4}$};
   \draw (2.8,1.2) node {\tiny $E_{6}$};
   \draw (3.8,1.2) node {\tiny $E_{8}$};
   \filldraw[fill=white]   (0,0) circle (3pt)
              (2,1) circle (3pt)
              (3,1) circle (3pt)
              (4,1) circle (3pt)
              (5,0) circle (3pt)
              (7,0) circle (3pt)
              (9,0) circle (3pt)
              (4,-1) circle (3pt)
              (6,-2) circle (3pt)
              (7,-2) circle (3pt);
   \filldraw  (1,1) circle (2pt)
              (3,2) circle (2pt)
              (5,2) circle (2pt)
              (4,2) circle (2pt)
              (1,0) circle (2pt)
              (2,0) circle (2pt)
              (3,0) circle (2pt)
              (4,0) circle (2pt)
              (6,0) circle (2pt)
              (8,0) circle (2pt)
              (10,1) circle (2pt)
              (11,1) circle (2pt)
              (1,-1) circle (2pt)
              (2,-1) circle (2pt)
              (3,-1) circle (2pt)
              (5,-2) circle (2pt)
              (7,-3) circle (2pt)
              (8,-3) circle (2pt)
              (5,-1) circle (2pt)
              (6,-1) circle (2pt)
              (7,-1) circle (2pt)
              (8,-1) circle (2pt)
              (6,1) circle (2pt)
              (7,1) circle (2pt)
              (8,1) circle (2pt);
 \end{tikzpicture}

\end{center}

This example has $56986$ jumping numbers in the interval $(0,2]$.
Using the extended version of Tucker's algorithm it takes $897.298 $
seconds to compute the whole list of jumping numbers and their
corresponding multiplier ideals. Using our method it only takes
$372.165$ seconds, i.e. it is roughly $9$ minutes faster.

\vskip 2mm

The main difference between the two algorithms stems in the fact
that Tucker needs to find first all the possible critical divisors.
We will see in the next section that our algorithm can be understood
as a method to find a unique and very precise contributing divisor.

\vskip 2mm

{ The input that we use in both algorithms, i.e. the log-resolution $\pi: X' \rightarrow X$ of an ideal 
$\fa \subseteq \cO_{X}$, is encoded using the intersection matrix and the vector of values for the divisor $F$ such
that $\fa\cdot\cO_{X'} = \cO_{X'}\left(-F\right)$. 
An algorithm to compute this data from a set of generators of the ideal $\fa$ has been described in \cite{ACAMB15}.
An implementation in {\tt Macaulay 2} will be available soon. For principal ideals this can be done using the 
{\tt Singular} \cite{GPS03}  package {\tt alexpoly.lib}.

}

\section{Jumping Divisors} \label{jumping_divisor}

The theory of critical divisors developed by Tucker \cite{Tuc10}
focuses on complete ideals very close to a given multiplier ideal.
%For example, a critical divisor $G$ associated to a jumping number $\lambda$ gives, as seen in Lemma \ref{adjacent},
%an adjacent ideal $$\pi_{*}\Oc_{X'}(\lceil K_\pi- \lambda F\rceil +G)\supseteq \J({\A}^{\lambda})$$ in the case of
%an $\fM$-primary ideal $\fa \subseteq \cO_{X,O}$.
The aim of this section is to understand the whole jump between two
consecutive multiplier ideals. To such purpose we introduce the
following natural definition:

\begin{definition}
 Let $ \lambda$ be a jumping numbers of an  ideal $\fa \subseteq \cO_{X,O}$.
 A reduced divisor  $G\leqslant F$ for which $\lambda$ is a candidate jumping number
 is called a {\it jumping divisor}
 for $\lambda$ if $$\J({\A}^{\lambda -\varepsilon})= \pi_{*}\Oc_{X'}(\lceil K_\pi- \lambda F\rceil +G),$$
 for $\varepsilon$ small enough.  We say that a jumping divisor is minimal if no proper subdivisor is a jumping divisor for $\lambda$, i.e.
$$\J({\A}^{\lambda -\varepsilon})\varsupsetneq \pi_{*}\Oc_{X'}(\lceil K_\pi- \lambda F\rceil+G')$$
for any $ 0\leqslant G'< G $.

\end{definition}

\begin{remark} \label{jd_contribution}
Any reduced divisor  $G\leqslant F$ for which $\lambda$ is a
candidate jumping number defines an ideal nested between two
consecutive multiplier ideals
$$\J({\A}^{\lambda -\varepsilon})\supseteq \pi_{*}\Oc_{X'}(\lceil K_\pi- \lambda F\rceil +G)\supseteq \J({\A}^{\lambda}).$$
Hence, a jumping divisor for $\lambda$ is a contributing divisor to
$\lambda$. In particular, a minimal jumping divisor can be
understood as  the minimal contribution which defines the preceding
multiplier ideal.
\end{remark}

\vskip 2mm

It is a striking fact that the methods used in the previous
section, in particular our main result Theorem \ref{lct_ideal}, will
allow us to construct the unique minimal jumping divisor associated
to a jumping number. In fact, we will see in Corollary
\ref{interval_jd} that the only jumping divisors are those reduced
divisors $D\leqslant F$ satisfying $G_\lambda \leqslant D \leqslant
H_\lambda$, where $G_\lambda$ and $H_\lambda$ are defined as
follows:

\begin{definition}
Let $ \lambda$ be a jumping number of an  ideal $\fa \subseteq
\cO_{X,O}$. Let $D_{{\lambda}-\varepsilon}=\sum
{e_i^{\lambda-\varepsilon}} E_i$ be the antinef closure of $\lfloor
(\lambda-\varepsilon) F - K_{\pi}\rfloor$ for $\varepsilon$ small
enough. Then we define:

\vskip 2mm

\begin{itemize}
 \item[$\cdot$] {\it Maximal jumping divisor:} Is the reduced divisor  $H_\lambda \leqslant F$ supported on those components
$E_i$ for which $\lambda e_i - k_i \in \bZ_{>0}$. Equivalently
 $$H_\lambda= \lceil K_\pi- (\lambda-\varepsilon) F\rceil - \lceil K_\pi- \lambda F\rceil.$$

 \item[$\cdot$] {\it Minimal jumping divisor:} Is the reduced divisor  $G_\lambda \leqslant F$ supported on those components
$E_i$ for which $$\lambda= \frac{k_i+1+
e_i^{\lambda-\varepsilon}}{e_i},$$ i.e.  supported on those divisors
where the minimum considered in Theorem \ref{lct_ideal} is achieved.
\end{itemize}

\end{definition}

\vskip 2mm

It is clear that $H_\lambda$ is a jumping divisor and $G_\lambda
\leqslant H_\lambda$. In fact, any reduced divisor $G \leqslant F$
that contributes to $\lambda$ satisfies $G \leqslant H_\lambda$. We
will prove next that $G_\lambda$ deserves the given name.

\vskip 2mm

\begin{proposition}
 Let $\lambda$ be a jumping number of an ideal $\fa \subseteq \cO_{X,O}$.
 The reduced divisor  $G_{\lambda}$ is a jumping divisor.
\end{proposition}

\begin{proof}
Since $G_\lambda \leqslant H_\lambda$, we have
$\lfloor(\lambda-\varepsilon) F - K_\pi\rfloor  \leqslant
\lfloor\lambda F - K_\pi\rfloor-G_{\lambda} $ and therefore
$$\J({\A}^{\lambda -\varepsilon})\supseteq \pi_{*}\Oc_{X'}(\lceil K_\pi- \lambda F\rceil +G_\lambda).$$

\vskip 2mm

For the reverse inclusion, let ${D}_{\lambda-\varepsilon}=\sum
e_i^{\lambda-\varepsilon} E_i$ be the antinef closure of
$\lfloor(\lambda-\varepsilon) F - K_\pi\rfloor$. We want to check
that $\lfloor\lambda F - K_\pi\rfloor-G_{\lambda} \leqslant
{D}_{\lambda-\varepsilon}$. To this purpose we only need to consider
the following cases:
\begin{itemize}
 \item[$\cdot$] If $E_i \leqslant G_\lambda$  then we have  $\lambda = \frac{k_i+1+e_i^{\lambda-\varepsilon}}{e_i}$.
 In particular  $ \lfloor \lambda e_i -k_i\rfloor -1 =  e_i^{\lambda-\varepsilon}$.
 \item[$\cdot$] If  $E_i \not\leqslant G_\lambda$ then we have $\lambda < \frac{k_i+1+e_i^{\lambda-\varepsilon}}{e_i}$.
 Thus $\lfloor \lambda e_i -k_i\rfloor < 1 + e_i^{\lambda-\varepsilon}$ and the
 result follows.
\end{itemize}
 \end{proof}

The unicity of the jumping divisor $G_\lambda$ is a consequence of
the following more general statement

 \begin{theorem} \label{jump1}
Let $  \lambda$ be a jumping number of an  ideal $\fa \subseteq
\cO_{X,O}$. Any contributing divisor $G\leqslant F$ associated to
$\lambda$ satisfies either:

\begin{itemize}
 \item[$\cdot$] $\J({\A}^{\lambda -\varepsilon})= \pi_{*}\Oc_{X'}(\lceil K_\pi- \lambda F\rceil +G)\varsupsetneq \J({\A}^{\lambda})$
 if and only if $G_\lambda\leqslant G$, or

  \item[$\cdot$] $\J({\A}^{\lambda -\varepsilon})\varsupsetneq \pi_{*}\Oc_{X'}(\lceil K_\pi- \lambda F\rceil+G)\varsupsetneq \J({\A}^{\lambda})$ otherwise.
\end{itemize}
\end{theorem}

\begin{proof}
Since $G \leqslant H_\lambda$, we have $\lfloor(\lambda-\varepsilon)
F - K_\pi\rfloor  \leqslant \lfloor\lambda F - K_\pi\rfloor-G $ and
therefore
$$\J({\A}^{\lambda -\varepsilon})\supseteq \pi_{*}\Oc_{X'}(\lceil K_\pi- \lambda F\rceil +G).$$

\vskip 2mm

Now assume $G_\lambda\leqslant G$. Then $\lfloor\lambda F -
K_\pi\rfloor-G \leqslant \lfloor\lambda F - K_\pi\rfloor-G_\lambda$,
and using the fact that $G_\lambda$ is a jumping divisor we obtain
the equality $\J({\A}^{\lambda -\varepsilon})=
\pi_{*}\Oc_{X'}(\lceil K_\pi- \lambda F\rceil +G).$

\vskip 2mm

If $G_\lambda \not\leqslant G$ we may consider a component
$E_i\leqslant G_\lambda$ such that $E_i\not\leqslant G$. Notice
that we have
$$v_i({D}_{\lambda-\varepsilon})=e_i^{\lambda-\varepsilon}= \lambda e_i-k_i-1 < \lambda e_i-k_i = v_i(\lfloor\lambda F - K_\pi\rfloor-G)$$
where ${D}_{\lambda-\varepsilon}=\sum e_i^{\lambda-\varepsilon} E_i$
is the antinef closure of $\lfloor(\lambda-\varepsilon) F -
K_\pi\rfloor$. Therefore, by Proposition \ref{semicont}, we get the
strict inclusion
$$\J({\A}^{\lambda -\varepsilon})\varsupsetneq \pi_{*}\Oc_{X'}(\lceil K_\pi- \lambda F\rceil+G).$$

\end{proof}

\begin{corollary}\label{unic_min_jd}
Let $\lambda$ be a jumping number of an  ideal $\fa \subseteq
\cO_{X,O}$. Then  $G_\lambda$ is the unique minimal jumping divisor
associated to $\lambda$.
\end{corollary}

Notice that Theorem \ref{jump1} also describes all the jumping
divisors associated to a given jumping number. Namely, we have

\begin{corollary}\label{interval_jd}
 Let $\lambda$ be a jumping number of an  ideal $\fa \subseteq \cO_{X,O}$. Then, any reduced divisor in the
 interval   $G_\lambda \leqslant D \leqslant H_\lambda$ is a jumping divisor.
\end{corollary}

\vskip 2mm

It is clear from its definition that maximal jumping divisors are
periodic, i.e. $H_{\lambda}= H_{\lambda + 1}$
 for any jumping number $\lambda$.  On the other hand, critical divisors do not satisfy any periodicity condition. One may
 find examples where  a divisor $G$ is a critical divisor for the jumping number $\lambda$ but not for $\lambda +1$ and vice versa.
 For minimal jumping divisors we have:

\begin{proposition}
Let $\lambda$ be a jumping number of an ideal $\fa \subseteq
\cO_{X,O}$ and  $G_{\lambda}$ its associated minimal jumping
divisor. Then we have:

\begin{itemize}
 \item [i)]  If $\lambda \leqslant 1$ then $G_\lambda \leqslant G_{\lambda + 1}$.
 \item [ii)] If $\lambda > 1$ then $G_\lambda = G_{\lambda + 1}$.
 \end{itemize}

\end{proposition}

\begin{proof}

Assume that there exists a prime divisor $E_i\leqslant G_\lambda$
such that $E_i\nleqslant G_{\lambda +1}$. Then, for a sufficiently
small $\varepsilon >0$ we have
$$\lambda = \frac{k_i+1+ e_i^{\lambda - \varepsilon}}{e_i} \hskip 5mm \text{and} \hskip 5mm\lambda +1 < \frac{k_i+1+ e_i^{(\lambda - \varepsilon)+1}}{e_i}$$
where
%\footnote{Let $\lambda' < \lambda$ be two consecutive jumping numbers. Notice that the jumping numbers $\lambda' +1 < \lambda +1$
%might not be consecutive. Is for this reason that we changed the notation for the values of the previous multiplier ideals.}
$D_{\lambda - \varepsilon}= \sum e_i^{\lambda - \varepsilon} E_i$
denotes the antinef closure of $\lfloor(\lambda - \varepsilon) F
-K_\pi\rfloor$ and equivalently, $D_{(\lambda - \varepsilon)+1}=
\sum e_i^{(\lambda - \varepsilon)+1} E_i$ is the antinef closure of
$\lfloor((\lambda - \varepsilon)+1) F - K_\pi\rfloor$.

\vskip 2mm

Therefore $$\frac{k_i+1+ e_i^{\lambda - \varepsilon}}{e_i} +1 <
\frac{k_i+1+ e_i^{(\lambda - \varepsilon)+1}}{e_i}$$ or equivalently
$e_i^{\lambda - \varepsilon} + e_i < e_i^{(\lambda -
\varepsilon)+1}$. Then we have $\A \cdot \J(\A^{\lambda -
\varepsilon}) \not \subseteq \J(\A^{(\lambda - \varepsilon)+1})$ so
we get a contradiction.

\vskip 2mm

For $\lambda > 1$ we have an equality $e_i^{\lambda - \varepsilon} +
e_i = e_i^{(\lambda - \varepsilon)+1}$ because of Skoda's theorem so
the result follows.
\end{proof}

\vskip 2mm

%\subsubsection{Other jumping properties}
Let $  {\lambda}' < \lambda$ be two consecutive jumping numbers of
an  ideal $\fa \subseteq \cO_{X,O}$. It is quite surprising that the
minimal jumping divisor $G_\lambda$ gives such nice approach to the
understanding of the jump from $\J(\A^{\lambda})$ to its preceding
multiplier ideal $\J(\A^{\lambda'})$. Taking into account that its
construction is based on Theorem \ref{lct_ideal}, where $\lambda$ is
obtained from the antinef divisor associated to $\J(\A^{\lambda'})$,
it would seem more natural to consider the jump in the other
direction. It turns out that the jump from $\J(\A^{\lambda'})$ to
$\J(\A^{\lambda})$ does not behave that nicely.

\begin{proposition} \label{jump2}
Let $  {\lambda}' < \lambda$ be two consecutive jumping numbers of
an  ideal $\fa \subseteq \cO_{X,O}$ and $D_{\lambda'}$ be the
antinef closure of $\lfloor\lambda' F - K_\pi \rfloor$. Then we
have:

\begin{itemize}
 \item[i)]  $\J(\A^{\lambda'})\varsupsetneq \pi_{*}\Oc_{X'}(-D_{\lambda'}-G_{\lambda})= \J(\A^{\lambda})$.

 \item[ii)] $\J(\A^{\lambda'})\varsupsetneq \pi_{*}\Oc_{X'}(\lceil K_\pi -(\lambda-\varepsilon) F \rceil - G_{\lambda})= \J(\A^{\lambda})$
\end{itemize}

\end{proposition}

\begin{proof}

 Let $D_{\lambda'}= \sum e_i^{\lambda'}E_i$, $D_{\lambda}= \sum
e_i^{\lambda}E_i$ be the antinef closures of $\lfloor \lambda' F -
K_\pi \rfloor$ and $\lfloor \lambda F - K_\pi \rfloor$ respectively.

\vskip 2mm

i) Since $G_\lambda$ is a jumping divisor we have $\lfloor \lambda F
- K_\pi \rfloor -G_\lambda \leqslant D_{\lambda'}$, and
 hence $\lfloor \lambda F - K_\pi \rfloor  \leqslant D_{\lambda'} + G_\lambda$. This gives the inclusion
 $\pi_{*}\Oc_{X'}(-D_{\lambda'}-G_{\lambda})\subseteq \J(\A^{\lambda})$.

 \vskip 2mm

In order to check the reverse inclusion
$\pi_{*}\Oc_{X'}(-D_{\lambda'}-G_{\lambda})\supseteq
\J(\A^{\lambda})$, it is enough, using Proposition \ref{semicont},
to prove $v_i(D_{\lambda'}+G_{\lambda}) \leqslant
v_i(D_{\lambda})=e_i^{\lambda}$ for any component $E_i$. We have
$e_i^{\lambda'} \leqslant e_i^{\lambda}$ just because
$\J(\A^{\lambda'})\varsupsetneq  \J(\A^{\lambda})$ and the
inequality is strict when $E_i \leqslant G_\lambda$, so the result
follows.

\vskip 2mm

ii) Let $D'$ be the antinef closure of $\lfloor (\lambda -
\varepsilon) F - K_\pi \rfloor + G_\lambda$. Since $G_\lambda
\leqslant H_\lambda$ we have
$$\lfloor (\lambda - \varepsilon) F - K_\pi \rfloor + G_\lambda \leqslant \lfloor \lambda F - K_\pi \rfloor \leqslant D_\lambda$$
so the inclusion $\pi_{*}\Oc_{X'}(\lceil K_\pi
-(\lambda-\varepsilon) F \rceil - G_{\lambda})\supseteq
\J(\A^{\lambda})$ holds. In order to prove the reverse inclusion we
will introduce an auxiliary divisor $D=\sum d_iE_i \in \Lambda$
defined as follows:
\begin{itemize}
   \item[$\cdot$]  $d_i= \lfloor (\lambda-\varepsilon)e_i-k_i\rfloor +1$  \hskip 2mm {\rm if} \hskip 2mm $E_i\leqslant G_\lambda$,
   \item[$\cdot$]  $d_i= e_i^{\lambda'}$ \hskip 33mm if  $E_i\leqslant H_\lambda$ but $E_i\not\leqslant G_\lambda$,
   \item[$\cdot$]  $d_i= \lfloor (\lambda-\varepsilon)e_i-k_i\rfloor $  \hskip 10mm otherwise.
  \end{itemize}

  \vskip 2mm

  Clearly we have  $\lfloor (\lambda - \varepsilon) F - K_\pi \rfloor + G_\lambda \leqslant D$, but we also have
$\lfloor \lambda  F - K_\pi \rfloor  \leqslant D$. Indeed,

  \vskip 2mm

\begin{itemize}
   \item[$\cdot$]  For $E_i\leqslant G_\lambda$ we have
   $\lfloor \lambda e_i -k_i \rfloor= \lambda e_i -k_i = \lfloor (\lambda-\varepsilon)e_i-k_i\rfloor +1= d_i$.

   \item[$\cdot$]  If $\lambda$ is a candidate for $E_i$ but $E_i\not\leqslant G_\lambda$, $\lfloor \lambda e_i -k_i \rfloor = \lambda e_i -k_i < 1 + e_i^{\lambda'}$,
   hence $\lfloor \lambda e_i -k_i \rfloor \leqslant e_i^{\lambda'}=d_i$.

   \item[$\cdot$]  Otherwise $\lfloor \lambda e_i -k_i \rfloor = \lfloor (\lambda-\varepsilon)e_i-k_i\rfloor = d_i$.
  \end{itemize}

    \vskip 2mm

  Therefore, taking antinef closures, we have $D' \leqslant D_{\lambda} \leqslant \widetilde{D}$. On the other hand $D\leqslant D'$.
  Namely, $v_i(D')\geqslant e_i^{\lambda'}$ at any $E_i$ because $\lfloor \lambda' F - K_\pi \rfloor \leqslant \lfloor (\lambda - \varepsilon) F - K_\pi \rfloor + G_\lambda$.
  Moreover, $v_i(D')\geqslant \lfloor (\lambda-\varepsilon)e_i-k_i\rfloor +\delta_i^{G_\lambda}$ by definition of antinef closure. Here,
  $\delta_i^{G_\lambda}=1$ if $E_i\leqslant G_\lambda$ and zero otherwise. Thus $v_i(D')\geqslant v_i(D)$ as desired.
 As a consquence $\widetilde{D} \leqslant D'$, which together with the previous  $D' \leqslant D_{\lambda} \leqslant \widetilde{D}$,
 gives $\widetilde{D} = D'=  D_{\lambda}$ and the result follows.
\end{proof}

\begin{remark}
 Contrary to the case of Theorem \ref{jump1}, $G_\lambda$ may not be minimal. In fact, we will see in Example \ref{ex_jump}
 a divisor $G<G_\lambda$ satisfying:
 $$\J(\A^{\lambda'})=\pi_{*}\Oc_{X'}(-D_{\lambda'})\varsupsetneq \pi_{*}\Oc_{X'}(-D_{\lambda'}-G)= \J(\A^{\lambda})\,.$$
\end{remark}

Despite the fact that the antinef closure of both
$\lfloor(\lambda-\varepsilon) F - K_\pi\rfloor$ and $\lfloor\lambda'
F - K_\pi\rfloor$ is $D_{\lambda'}$, it is quite remarkable that the
above jumping property does not hold taking $\lfloor\lambda' F -
K_\pi\rfloor$,
 i.e. the equality $\pi_{*}\Oc_{X'}(\lfloor\lambda' F - K_\pi\rfloor-G_{\lambda})= \J(\A^{\lambda})$
 is not always true.

\subsection{Invariance of the minimal jumping divisor with respect to the log-resolution}

Multiplier ideals and jumping numbers are known to be independent of
the chosen log-resolution of the initial ideal $\fa \subseteq
\Oc_{X,O}$. The aim of this section is to prove that the minimal
jumping divisor is {\it generically} independent of the
log-resolution in a sense that we will make precise below. As a
consequence of Proposition \ref{min_min} and Corollary
\ref{critical_min} in Section \S5, critical divisors will also be
generically independent of the log-resolution. This is a remarkable
fact since, as it was pointed out by Tucker in \cite[Remark
3.4]{Tuc10}, there is no reason to believe that critical divisors
(and by extension minimal jumping divisors) are independent of the
resolution since they depend on all the divisorial valuations
appearing in $F$.

\vskip 2mm

We start fixing some notation that we will use in this section. Let
$\pi':X'\lra X$ be the minimal log-resolution of an ideal $\fa
\subseteq  \Oc_{X,O}$. Any other log-resolution $\pi:Y\lra X$
factors through $\pi'$, i.e. there is a birational morphism $g:Y\lra
X'$ such that $\pi= \pi' \circ g$ (see \cite[Theorem 4.1]{Lip69}).

\vskip 2mm

For a given jumping number $\lambda$ of $\fa$ we will denote
$G'_\lambda$ the minimal jumping divisor of $\pi'$ and
$E'_1,\dots,E'_r$ the exceptional components of $E'=Exc(\pi')$. If
$G_\lambda$ and $E_1,\dots,E_s$ are the minimal jumping divisor and
the exceptional components of $E=Exc(\pi)$ for any other
log-resolution $\pi$, we will enumerate them setting $E_i$ equal to
the strict transform by $g$ of $E'_i$ for $1\leqslant i \leqslant
r$. If no confusion arise, we will use the same symbol to denote a
divisor $D=\sum_{i=1}^r d_i E'_i$ on $X'$ or its strict transform
$D=\sum_{i=1}^r d_i E_i$ on $Y$.

\begin{theorem} \label{invariance}
With the previous notations, $G_\lambda$ is independent of the
log-resolution $\pi$  if and only if $\pi$ does not include any
blowing-up at points in the intersection of two components of the
minimal jumping divisor $G'_\lambda$ of the minimal log-resolution.
\end{theorem}

Actually, from the proof of this result, we can express the minimal
jumping divisor of any resolution. To such purpose we need to fix
some notation:

\vskip 2mm

 A reduced divisor with exceptional support
$D=E_{i_1}+\cdots+E_{i_m} \leqslant E$ is a {\it chain with ends}
$E_{i_1}$ and $E_{i_m}$ if $a_D(E_{i_1})=a_D(E_{i_m})=1$ and
$a_D(E_{i_k})=2$ for any other $1<k<m$. Given $E_{j_1}, E_{j_2}
\leqslant E$, we say that the chain above {\it connects}  $E_{j_1}$ and
$E_{j_2}$ if $E_{j_1}\in \adj\left(E_{i_1}\right)$  and  $E_{j_2}\in
\adj\left(E_{i_m}\right)$. Observe that if $E_{j_1}$ and  $E_{j_2}$
are adjacent in $E$, a chain connecting them will be $D=0$.

\begin{corollary}
Keeping the above notations we have

\begin{equation} \label{invariance_divisor}
G_\lambda= G'_\lambda + \sum_{\stackrel{E'_i+E'_j \leqslant G'_\lambda}{E'_i\cdot E'_j=1}} D_{ij}
\end{equation}
where $D_{ij}$ is a chain connecting $E_i$ and $E_j$.
\end{corollary}

{ Consider {\it generic} log-resolutions as those obtained from a
minimal one by further blowing-ups at simple (and hence generic) points on the
exceptional components. Then, Theorem \ref{invariance} states that generic log-resolutions
have the same minimal jumping divisor.} This {\it generictiy} may be formulated, when $X$
is smooth, in terms of valuations in the valuative tree
$\mathcal{V}$ of Favre-Jonsson \cite{FJ04}. Consider the dual graphs
$\Gamma$ and $\Gamma'$ of $E$ and $E'$ respectively, embedded in the
valuative tree $\mathcal{V}$ as in \cite[Chapter 6]{FJ04} and let
$\nu_i$ denote the divisorial valuation centered at $E_i$.

\begin{corollary}
The minimal jumping divisor $G_\lambda$ of  $\pi$ equals the minimal
jumping divisor $G'_\lambda$ if and only if $\Gamma$ has no vertex
inside any segment $]\nu_i,\nu_j[$ for which $E'_i$ and $E'_j$ are
adjacent in $E'$ and belong to $G'_\lambda$.
\end{corollary}

\vskip 2mm

{\it Proof of Theorem \ref{invariance}.} \hskip 2mm 
{
Let $\lambda' < \lambda$ be two consecutive jumping numbers of $\fa$. 
We will argue by induction on the number of blowing-ups needed to reach $Y$
from a minimal resolution. In order to simplify the notation, we will assume 
throughout this proof that $X'$ also dominates a minimal log-resolution and
that $Y$ is obtained from $X'$ by one blowing-up $g:Y\lra X'$ at a closed point 
$p\in X'$ giving the exceptional component $E_s$. Assume that (\ref{invariance_divisor}) 
holds on $X'$ and let us prove it on $Y$. Notice that, keeping the notation 
used in this section, we are in the case $r+1=s$.}

\vskip 2mm

Let $F'=\sum_{i=1}^r e_i E'_i$ and $F=\sum_{i=1}^s e_i E_i$ be the
divisors in $X'$ and $Y$ respectively such that $\fa \cO_{X'}=
\cO_{X'}(-F')$ and $\fa \cO_{Y}= \cO_{Y}(-F)$. We also consider the
antinef divisors $D'_{\lambda'}=\sum_{i=1}^r e_i^{\lambda'} E'_i$
and $D_{\lambda'}=\sum_{i=1}^s e_i^{\lambda'} E_i$ for which
$\J(\fa^{\lambda'})= \pi'_* \Oc_{X'}(-D'_{\lambda'})= \pi_*
\Oc_{Y}(-D_{\lambda'})$ sharing the first $r$ coefficients since
multiplier ideals are independent of the log-resolution. Moreover,
by Theorem \ref{lct_ideal}
$$\lambda=
\min_{1\leqslant i \leqslant r}\left\{\frac{k_i+1+
e_i^{\lambda'}}{e_i}\right\}= \min_{1\leqslant i \leqslant
s}\left\{\frac{k_i+1+ e_i^{\lambda'}}{e_i}\right\},$$ clearly
demonstrating { that the strict transform of $G'_\lambda$ is contained in $G_\lambda$. In particular,
$\lambda e_i-k_i = 1+ e_i^{\lambda'}$ if and only if $E_i\leqslant
G_\lambda$ and $\lambda e_i-k_i < 1+ e_i^{\lambda'}$ otherwise.

\vskip 2mm

We distinguish two cases:

\vskip 2mm

i) The closed point $p$ lies only on one exceptional divisor
$E'_j$. Then we have $e_s=e_j$, $k_s=k_j+1$ and
$e_s^{\lambda'}=e_j^{\lambda'}$ and thus $$v_s(\lfloor \lambda F
-K_\pi\rfloor)=\lfloor \lambda e_s -k_i\rfloor=\lfloor \lambda e_j
-k_j\rfloor - 1 \leqslant  e_j^{\lambda'}=e_s^{\lambda'}.$$ Hence
$E_s$ can not belong to $G_\lambda$.

\vskip 2mm

ii) The closed point $p$ lies on the intersection of two
exceptional divisors $E'_{j_1}$ and $E'_{j_2}$. Then we have
$e_s=e_{j_1}+e_{j_2}$, $k_s=k_{j_1}+k_{j_2}+1$ and
$e_s^{\lambda'}=e_{j_1}^{\lambda'}+ e_{j_2}^{\lambda'}$ so
$$v_s(\lfloor \lambda F -K_\pi\rfloor)=\lfloor \lambda e_s
-k_s\rfloor=\lfloor \lambda e_{j_1} -k_{j_1} + \lambda e_{j_2}
-k_{j_2}\rfloor - 1 \leqslant
e_{j_1}^{\lambda'}+e_{j_2}^{\lambda'}+1=e_s^{\lambda'}+1,$$ and
equality holds if and only if $E'_{j_1}+E'_{j_2}\leqslant G_\lambda$.
In particular, $E_s$ does not belong to $G_\lambda$ whenever none or
just one of the components $E'_{j_1},E'_{j_2}$ belong to
$G'_\lambda$.\qed}

\subsection{Geometric properties of minimal jumping divisors in the dual graph}

Assume that a critical divisor $G$ associated to a jumping number
$\lambda$ has exceptional support. One of the key ingredients in
Tucker's algorithm for the computation of jumping numbers is that
$G$ satisfies some nice geometric conditions when viewed in the dual
graph: $G$ is a connected chain and its ends must be either rupture
or dicritical divisors (see Proposition \ref{T2}). Then, it is
natural to ask whether jumping divisors satisfy analogous
properties.

\vskip 2mm

Throughout this section we will also assume that the minimal jumping
divisor $G_\lambda$ has exceptional support. Then, it may have
several connected  components in the dual graph and these components
are not necessarily  chains. However, we can still control the ends
of each component. To prove the main result of this section (see
Theorem \ref{geo_dual_graph}) we need some preliminary results
first. Keep the notations of Section \S\ref{preli}.

\begin{lemma} \label{num}
Let $\lambda$ be a jumping number of an ideal $\fa \subseteq
\cO_{X,O}$. For any component $E_i$ of the minimal jumping divisor
$G_{\lambda}$ we have
$$\left(\left\lceil K_\pi-\lambda  F\right\rceil + G_{\lambda}\right)\cdot E_i = -2 + \lambda \rho_i +
\sum_{ E_j \in \adj(E_i)} \left\{\lambda e_j-k_j\right\} +
a_{G_{\lambda}}\left(E_i\right).$$
\end{lemma}
\begin{proof}
For any $E_i\leqslant G_{\lambda}$  we have
\begin{multline*}
(\left\lceil K_{\pi}-\lambda F\right\rceil + G_{\lambda})\cdot E_i =
(\left(K_\pi-\lambda F\right)+\left\{-K_\pi+
\lambda F\right\} + G_{\lambda} - E_i + E_i)\cdot E_i = \\
= \left(K_{\pi}+ E_i\right)\cdot E_i - \lambda F\cdot E_i +
\left\{\lambda F-K_{\pi}\right\}\cdot E_i +
\left(G_{\lambda}-E_i\right)\cdot E_i.
\end{multline*}
Let us now compute each summand separately. Firstly, the adjunction
formula gives $\left(K_{\pi}+E_i\right)\cdot E_i = -2$ because $E_i
\cong \bP^1$. As for the second and fourth terms, the equality
$-\lambda F \cdot E_i = \lambda\rho_i$ follows from the definition
of the excesses, and clearly $a_{G_{\lambda}}\left(E_i\right) =
\left(G_{\lambda}-E_i\right)\cdot E_i$ because $E_i \leqslant
G_{\lambda}$.

Therefore it only remains to prove that
\begin{equation} \label{eq-decimal-part}
\left\{\lambda F-K_{\pi}\right\}\cdot E_i = \sum_{ E_j \in
\adj(E_i)} \left\{\lambda e_j-k_j\right\},
\end{equation}
which is also quite immediate. Indeed, writing
$$\left\{\lambda F-K_{\pi}\right\} = \sum_{j=1}^r\left\{\lambda e_j-k_j\right\}E_j,$$
equality (\ref{eq-decimal-part}) follows by observing that (for $j
\neq i$), $E_j \cdot E_i = 1$ if and only if $E_j \in
\adj\left(E_i\right)$, and the term corresponding to $j=i$ vanishes
because we have $\lambda e_i-k_i \in \bZ$.
\end{proof}

\begin{remark}
 It is important to notice that $( \lceil K_\pi-\lambda F\rceil + G_\lambda)\cdot E_i\in \Z$, that is
 $-2 + \sum_{ E_j \in \adj(E_i)} \left\{ \lambda e_j - k_j\right\}  + \lambda \rho_i + a_{G_\lambda}(E_i)\in \Z$.
\end{remark}

The following result is an analogue of the numerical conditions that
critical divisors satisfy (see Proposition \ref{Num_condition2}).
Unfortunately it does not provide a characterization of minimal
jumping divisors.

\begin{proposition} \label{Num_condition}
Let $\lambda$ be a jumping number of an  ideal $\fa \subseteq
\cO_{X,O}$. For any component $E_i\leqslant G_\lambda$ of the
minimal jumping divisor $G_{\lambda}$ we have
$$\left(\left\lceil K_{\pi}- \lambda F\right\rceil +
G_{\lambda}\right)\cdot E_i\geqslant 0.$$

\end{proposition}

\begin{proof}

Let $G_\lambda$ be the minimal jumping divisor. Given a prime
divisor $E_i\leqslant G_\lambda$ we consider the short exact
sequence
\begin{multline*}
0 \longrightarrow \Oc_{X'}\left(\left\lceil K_{\pi}- \lambda F\right\rceil + G_{\lambda} - E_i\right) \longrightarrow \Oc_{X'}\left(\left\lceil K_{\pi}- \lambda F\right\rceil + G_{\lambda}\right) \longrightarrow \\
\longrightarrow \Oc_{E_i}\left(\left\lceil K_{\pi}- \lambda
F\right\rceil + G_{\lambda}\right) \longrightarrow 0
\end{multline*}
Pushing it forward to $X$ we get
\begin{multline*}
0 \longrightarrow \pi_*\Oc_{X'}\left(\left\lceil K_{\pi}- \lambda F\right\rceil + G_{\lambda} - E_i\right) \longrightarrow \pi_*\Oc_{X'}\left(\left\lceil K_{\pi}- \lambda F\right\rceil + G_{\lambda}\right) \longrightarrow \\
\longrightarrow H^0\left(E_i,\Oc_{E_i}\left(\left\lceil K_{\pi}-
\lambda F\right\rceil + G_{\lambda}\right)\right) \otimes \bC_O,
\end{multline*}
where $\bC_O$ denotes the skyscraper sheaf supported at $O$ with
fibre $\bC$. The minimality of $G_\lambda$  (see Theorem
\ref{jump1}) implies that
$$\pi_*\Oc_{X'}\left(\left\lceil K_{\pi}- \lambda F\right\rceil + G_{\lambda} - E_i\right) \neq \pi_*\Oc_{X'}\left(\left\lceil K_{\pi}-
\lambda F\right\rceil + G_{\lambda}\right).$$ Thus
$H^0\left(E_i,\Oc_{E_i}\left(\left\lceil K_{\pi}- \lambda
F\right\rceil + G_{\lambda}\right)\right) \neq 0$, or equivalently
$\left(\left\lceil K_{\pi}- \lambda F\right\rceil +
G_{\lambda}\right)\cdot E_i\geqslant 0$.
\end{proof}

With the above ingredients we can provide the following geometric
property of minimal jumping divisors when viewed in the dual graph.

\begin{theorem}\label{geo_dual_graph}
Let $G_\lambda$ be the minimal jumping divisor associated to a
jumping number $\lambda$ of an ideal $\fa \subseteq \cO_{X,O}$. Then
the ends of a connected component of $G_\lambda$ must be either
rupture or dicritical divisors.
\end{theorem}

\begin{proof}
Assume that an end $E_i$ of a connected component of $G_\lambda$ is
neither a rupture nor a dicritical divisor. It means that $E_i$ has
no excess, i.e. $\rho_i=0$, and that it has one or two adjacent
divisors, say $E_j$ and $E_l$, in the dual graph but at most one of
them belongs to $G_\lambda$.

\vskip 2mm

For the case that $E_i$ has two adjacent divisors $E_j$ and $E_l$
the formula given in lemma \ref{num} reduces to $( \lceil K_\pi-
\lambda F\rceil + G_\lambda)\cdot E_i= -2 + \left\{ \lambda e_j
-k_j\right\} + \left\{ \lambda e_l -k_l\right\}+ \lambda \rho_i +
a_{G_\lambda}(E_i).$ Then:

\vskip 2mm

\begin{itemize}
\item[$\cdot$] If $E_i$ has valence one in $G_\lambda$, e.g. $E_l \not \leqslant
G_\lambda$ then $$(\lceil K_\pi- \lambda F\rceil + G_\lambda)\cdot
E_i= -2 + \left\{ \lambda e_l -k_l \right\}+ 1 < 0. $$

\vskip 2mm

\item[$\cdot$] If $E_i$ is an isolated component of $G_\lambda$, i.e., $E_j, E_l \not \leqslant
G_\lambda$ then $$( \lceil K_\pi- \lambda F\rceil + G_\lambda)\cdot
E_i= -2 + \left\{ \lambda e_j -k_j\right\} + \left\{ \lambda e_l
-k_l\right\} <0. $$
\end{itemize}

\vskip 2mm

If $E_i$ has just one adjacent divisor $E_j$, i.e. $E_i$ is an end
of the dual graph, the formula reduces to  $( \lceil  K_\pi-\lambda
F\rceil + G_\lambda)\cdot E_i= -2 + \left\{ \lambda e_j -k_j\right\}
+ \lambda \rho_i + a_{G_\lambda}(E_i)$. Then:

\begin{itemize}
\item[$\cdot$] If $E_i$ has valence one in $G_\lambda$ then $( \lceil K_\pi- \lambda F\rceil + G_\lambda)\cdot E_i= -2 + 1 < 0 $

\vskip 2mm

\item[$\cdot$] If $E_i$ is an isolated component of $G_\lambda$ then $$( \lceil K_\pi- \lambda F\rceil + G_\lambda)\cdot E_i= -2 + \left\{
\lambda e_j -k_j\right\}  <0 .$$
\end{itemize}

\vskip 2mm

In any case we get a contradiction with Proposition
\ref{Num_condition}.
\end{proof}

\begin{remark}
It follows from \cite[Theorem 3.3]{Vey} that the minimal jumping
divisor associated to the log-canonical threshold is connected in
the case that $X$ is smooth.
\end{remark}

As a consequence  we may also give  the following refinement of
Proposition \ref{Num_condition}.

\begin{proposition} \label{Num_condition2}
Let $\lambda$ be a jumping number of an $\fM$-primary ideal $\fa
\subseteq \cO_{X,O}$. If $E_i\leqslant G_\lambda$ is neither a
rupture nor a dicritical component of the minimal jumping divisor
$G_{\lambda}$ we have $$\left(\left\lceil K_{\pi}- \lambda
F\right\rceil + G_{\lambda}\right)\cdot E_i= 0.$$
\end{proposition}

\begin{proof}
Assume that $E_i\leqslant G_{\lambda}$ is neither a rupture or a
dicritical component. In particular, it is not the end of a
connected component of $G_{\lambda}$. Thus, $E_i$ has exactly two
adjacent components $E_j$ and $E_l$ in $G_{\lambda}$, and its excess
is $\rho_i=0$. The formula given in Lemma \ref{num} reduces to
$$\left(\left\lceil K_{\pi}- \lambda F\right\rceil + G_{\lambda}\right)\cdot E_i  = -2 + \lambda \rho_i + \left\{\lambda e_j-k_j\right\} + \left\{\lambda e_l-k_l\right\} + a_{G_{\lambda}}\left(E_i\right).$$
Notice that $a_{G_\lambda}(E_i)=2$, and also $\left\{\lambda
e_j-k_j\right\} = \left\{\lambda e_l-k_l\right\} = 0$ because $E_j$
and $E_l$ are components of $G_{\lambda}$, so finally
$\left(\left\lceil K_{\pi}- \lambda F\right\rceil +
G_{\lambda}\right)\cdot E_i = 0$.
\end{proof}

\section{Minimal contributing divisors} \label{mini}

The theory of minimal jumping divisors introduced in Section
\S\ref{jumping_divisor} can be included in a more general framework
that we will describe in this section. To such purpose we will give
our own perspective of the work of Hyry-J\"arviletho \cite{HJ11} and
its relation with the theory of contributing divisors of Tucker
\cite{Tuc10}.

\vskip 2mm

Let  $\lambda$ be a jumping number of an  ideal $\fa \subseteq
\cO_{X,O}$. Recall that  a reduced divisor  $G\leqslant F$ that
contributes to $\lambda$ defines an ideal nested between two
consecutive multiplier ideals
$$\J({\A}^{\lambda -\varepsilon})\supseteq \pi_{*}\Oc_{X'}(\lceil K_\pi- \lambda F\rceil +G)\varsupsetneq \J({\A}^{\lambda}).$$
We may interpret that $\lambda$ is {\it parametrized} by the set of
nested ideals defined by contributions but this is far from being a
one-to-one correspondence. An easy way to detect such a nested ideal
is finding a suitable critical divisor using Tucker's algorithm. The
approach given in the previous sections is more economical in the
sense that each jumping number is parametrized by its unique minimal
jumping divisor $G_\lambda$ or equivalently, its preceding
multiplier ideal. \vskip 2mm

Hyry-J\"arviletho \cite{HJ11} give a similar approach where jumping
numbers are parametrized by general antinef
divisors\footnote{Hyry-J\"arviletho only consider the case of
$\fM$-primary ideals on smooth surfaces and consequently antinef
divisors with exceptional support but their ideas also hold in
general}, or equivalently complete ideals not necessarily nested in
the chain of multiplier ideals. We should point out that their
results also hold for the case that $X$ has rational singularities
since their arguments are based on divisorial considerations. Given
any antinef divisor
 $D=\sum d_i E_i \in {\rm Div}(X')$, they considered the following notions:

 \vskip 2mm

\begin{itemize}
 \item[$\cdot$] {\it Jumping number corresponding to $D$:} $$\lambda_D:= \min_i\left\{\frac{k_i+1+ d_i}{e_i}\right\}. $$

 \item[$\cdot$] {\it Support of a jumping number corresponding to $D$:}
 $$S_D:=\left\{ i \hskip 2mm | \hskip 2mm \lambda_D=\frac{k_i+1+ d_i}{e_i} \right\}.$$

 \item[$\cdot$] {\it Contributing divisor associated to $D$:} $$G_D:= \sum_{i\in S_D} E_i.$$

\end{itemize}

\vskip 2mm

Hyry-J\"arviletho proved in \cite[Proposition 1]{HJ11} that all
jumping numbers of $\fa$ can be obtained in this way: as $\lambda_D$
for a suitable antinef divisor $D\in {\rm Div}(X')$ (or equivalently
a complete ideal $I_D$). Moreover, they give in \cite[Theorem
1]{HJ11} a combinatorial criterion that detects the existence of
such antinef divisors. The simplest parametrizations they used to
describe the set of jumping numbers are given by antinef divisors
corresponding to critical divisors (see \cite[Theorem 2]{HJ11}).

\vskip 2mm

In general, the complete ideal $I_D$ associated to an antinef
divisor $D\in {\rm Div}(X')$ satisfies $\J({\A}^{\lambda_D
-\varepsilon})\supseteq I_D$  but does not necessarily contain
$\J({\A}^{\lambda_D})$. However, if $I_D$ is nested in between two
consecutive multiplier ideals
$$\J(\A^{\lambda-\varepsilon})\supseteq I_D \varsupsetneq
\J(\A^{\lambda})$$ then it must satisfy $\lambda=\lambda_D$.

\begin{remark}
One can also interpret this framework through the  generalized
version of log-canonical thresholds already introduced by
J\"arviletho in \cite{Jar11}. Namely, the  {\it log-canonical
threshold} with respect to any other ideal $\fb  \subseteq
\cO_{X,O}$ is defined as follows:
$${\rm lct}_{\fb}(\fa):= \inf \{ c\in \bQ_{>0} \hskip 2mm | \hskip 2mm \J(\A^{c}) \not\supset \fb \}$$
Notice that whenever $I_D$ is the complete ideal associated  to an
antinef divisor $D\in {\rm Div}(X')$, then $\lambda_D= {\rm
lct}_{I_D}(\fa)$.
\end{remark}

Hyry-J\"arviletho \cite[Lemma 11]{HJ11} proved that if $D\in {\rm
Div}(X')$ is an antinef divisor then $G_D$ is a contributing divisor
for $\lambda_D$. In fact, the contributing divisors obtained in this
way satisfy some nice properties as we will see next.

\begin{proposition}\label{G_DleqG}
Let $G$ be a  contributing divisor associated to a jumping number
$\lambda$. Let $D$ be the antinef closure of $\lfloor\lambda F -
K_\pi\rfloor-G $. Then $G_D \leqslant G$.
 \end{proposition}

 \begin{proof}
Let $D=\sum d_i E_i$   be the antinef closure of $\lfloor\lambda F -
K_\pi\rfloor-G $. Since $I_D$ is a nested ideal in the chain of
multiplier ideals, then we have
$$\lambda=\lambda_D=\min_i \left\{\ \frac{k_i+1+d_i}{e_i} \right\}.$$ Hence $\lambda e_i -k_i\leqslant 1+d_i$ and
equality holds if and only if $i\in S_D$. In order to prove $G_D
\leqslant G$ we will show that $E_i\not\leqslant G$ implies
$E_i\not\leqslant G_D$. Indeed, if $E_i\not\leqslant G$ and
$E_i\leqslant G_D$ then $\lfloor \lambda e_i -k_i\rfloor \leqslant
d_i$ (just because  $\lfloor \lambda F - K_\pi \rfloor - G \leqslant
D$ by Lemma \ref{def_antinef}) and $\lambda e_i -k_i- 1=d_i$ so we
get a contradiction.
 \end{proof}

 \begin{proposition}\label{Prop3}
 Let $\lambda=\lambda_{D'}$ be a jumping number associated to an antinef divisor $D'\in {\rm Div}(X')$. Let $D$ be the antinef
 closure of $\lfloor\lambda F - K_\pi\rfloor-G_{D'} $. Then we have $D \leqslant D'$, $\lambda_D=\lambda_{D'}$, $S_D=S_{D'}$ and $G_D=G_{D'}$.
 \end{proposition}

 \begin{proof}
 Using the definition of antinef closure (see Lemma \ref{def_antinef}), in order to get $D \leqslant D'$
 we only need to prove that $\lfloor \lambda F - K_\pi \rfloor - G_{D'} \leqslant D'$.  Set $D'=\sum d'_i E_i$. By hypothesis
  $$\lambda=\lambda_{D'}=\min_i \left\{\ \frac{k_i+1+d'_i}{e_i} \right\}$$ therefore we have
  $\lfloor \lambda e_i -k_i\rfloor \leqslant d'_i$ if $i\not\in S_{D'}$, whereas $\lfloor \lambda e_i -k_i\rfloor -1= d'_i$
  if $i\in S_{D'}$ as desired.

  \vskip 2mm

Notice then  that we have $ \J(\A^{\lambda-\varepsilon})\supseteq
I_D\supseteq I_{D'}$ so, given the fact that
 $I_{D'}\not \subseteq \J(\A^{\lambda})$, we get $\lambda_D=\lambda$. Now, the inclusion of divisors $D \leqslant D'$
 having the same minimum $\lambda_D=\lambda_{D'}$, gives the inclusion of supports $S_D\supseteq S_{D'}$ and equivalently $G_D\geqslant G_{D'}$.
 On the other hand, taking $G=G_{D'}$ in Proposition \ref{G_DleqG}, we get the reverse inequality of divisors
 $G_D\leqslant G_{D'}$ so we are done.
 \end{proof}

The main result of this section is that we can find a minimal
contributing divisor among all contributing divisors defining the
same nested ideal.

 \begin{theorem}\label{minimal_contributing}
Let $G$ be a  contributing divisor associated to a jumping number
$\lambda$. Let $D$ be the antinef closure of $\lfloor\lambda F -
K_\pi\rfloor-G $, which gives a nested ideal
$$\J(\A^{\lambda-\varepsilon})\supseteq I_D = \pi_*\Oc_{X'}\left(\left\lceil K_{\pi}-
\lambda F\right\rceil + G\right)  \varsupsetneq
\J(\A^{\lambda}).$$ Then we also have $I_D =
\pi_*\Oc_{X'}\left(\left\lceil K_{\pi}- \lambda F\right\rceil +
G_D\right)$. Furthermore, $G_D$ is the minimal contributing divisor
associated to $\lambda$ that defines the same ideal $I_D$, that is:

\begin{itemize}
 \item[$\cdot$] Any contribution  $G'$ to $\lambda$ defining $I_D = \pi_*\Oc_{X'}\left(\left\lceil K_{\pi}-
\lambda F\right\rceil + G'\right)$ must satisfy $G_D \leqslant G'$.

 \item[$\cdot$] Any proper subdivisor $G'< G_D$ defines an strictly included ideal $$I_D
 \varsupsetneq \pi_*\Oc_{X'}\left(\left\lceil K_{\pi}- \lambda F\right\rceil + G'\right).$$
\end{itemize}

 \end{theorem}

 \begin{proof}
 Let $D'$ be
the antinef closure of $\lfloor\lambda F - K_\pi\rfloor-G_D $. We
will see first that $D=D'$ thus giving the desired equality of
ideals $$I_D = \pi_*\Oc_{X'}\left(\left\lceil K_{\pi}- \lambda
F\right\rceil + G\right)=\pi_*\Oc_{X'}\left(\left\lceil K_{\pi}-
\lambda F\right\rceil + G_D\right)=I_{D'}.$$ In virtue of
Proposition \ref{G_DleqG}, we have $G_D \leqslant G$ so
$\lfloor\lambda F - K_\pi\rfloor-G \leqslant \lfloor\lambda F -
K_\pi\rfloor-G_D $ and $D\leqslant D'$. The reverse inequality
$D\geqslant D'$ is a consequence of Proposition \ref{Prop3}.

\vskip 2mm

To show that $G_D$ is the minimal contributor to the jumping number
$\lambda$ that defines the same ideal $I_D$ we will prove the
following equivalent result:

\vskip 2mm

{\it Claim: } Any contributor  $G'$ to $\lambda$ for which
  $I_D\supseteq \pi_*\Oc_{X'}\left(\left\lceil K_{\pi}-
\lambda F\right\rceil + G'\right)$  also satisfies the reverse
inclusion $I_D\subseteq \pi_*\Oc_{X'}\left(\left\lceil K_{\pi}-
\lambda F\right\rceil + G'\right)$ if and only if $G_D\leqslant G'$.

\vskip 2mm

{\it Proof of Claim:} Suppose first that $G_D\leqslant G'$. Then
$\lfloor\lambda F - K_\pi\rfloor-G' \leqslant \lfloor\lambda F -
K_\pi\rfloor-G_D $ and hence  $D''\leqslant D'=D$, where $D''$ is
the antinef closure of $\lfloor\lambda F - K_\pi\rfloor-G'$.
Therefore $I_D \subseteq I_{D''}$ as wanted.

\vskip 2mm

Assume now that  $G_D\not\leqslant G'$ and pick a component
$E_i\leqslant G_D$ such that $E_i\not\leqslant G'$. By hypothesis
$I_D \supseteq I_{D''}$ and equivalently $D\leqslant D''$ but in
fact $D< D''$ since
 $$v_i(D)=\lambda e_i - k_i -1 < \lambda e_i - k_i=v_i(\lfloor\lambda F - K_\pi\rfloor-G' )\leqslant v_i(D'').$$
The result follows then from Proposition \ref{semicont}.
 \end{proof}

 It turns out that critical divisors are also minimal in the above sense as we can see in the
 following generalization of \cite[Proposition 3]{HJ11}.

 \begin{corollary} \label{critical_min}
Let $G$ be a  contributing divisor associated to a jumping number
$\lambda$. Let $D$ be the antinef closure of $\lfloor\lambda F -
K_\pi\rfloor-G $. Then $G$ is a critical divisor if and only if $G_D
= G$ and $I_D$ and $ \J(\A^{\lambda})$ do not admit strictly nested
ideals between them defined by contributors to $\lambda$.
 \end{corollary}

\begin{proof}
Assume first that  $G_D = G$. Then, by Theorem
\ref{minimal_contributing}, any proper subdivisor $0\leqslant G'<G$
defines an ideal strictly included in $I_D \varsupsetneq
\pi_*\Oc_{X'}\left(\left\lceil K_{\pi}- \lambda F\right\rceil +
G'\right) \supseteq  \J(\A^{\lambda}) $. Since $I_D$ and $
\J(\A^{\lambda})$  do not admit strictly nested ideals between them
coming from contributors, we get $\pi_*\Oc_{X'}\left(\left\lceil
K_{\pi}- \lambda F\right\rceil + G'\right) =  \J(\A^{\lambda})$ so
$G$ is a critical divisor.

\vskip 2mm

Assume now that $G$ is a critical divisor. By Proposition
\ref{G_DleqG} we have $G_D\leqslant G$. Both divisors define the
same ideal by Theorem \ref{minimal_contributing} so they must be
equal otherwise we would have a contradiction with the fact that $G$
is a critical divisor.

\vskip 2mm

Finally we will see that there is no contributing divisor $G'$
associated to $\lambda$ defining a strictly nested ideal $$I_D
\varsupsetneq  \pi_*\Oc_{X'}\left(\left\lceil K_{\pi}- \lambda
F\right\rceil + G'\right)  \varsupsetneq \J(\A^{\lambda}).$$

Assume that such $G'$ exists and let $D'$ be the antinef closure of
$ \lfloor \lambda F - K_\pi\rfloor - G'$. Then the inclusion of
divisors $D<D'$ having the same minimum
$\lambda_D=\lambda_{D'}=\lambda$ implies $S_{D'} \subseteq S_D$ and
$G_{D'} \leqslant G_D$. Since $G=G_D$ is minimal, applying Theorem
\ref{minimal_contributing}, we must have $G=G_D=G_{D'}\leqslant G'$
contradicting the starting hypothesis of inclusion of ideals.
\end{proof}

 The minimal jumping divisor introduced in Section \S\ref{jumping_divisor} fits nicely in this theory.
Given a jumping number $\lambda$ of an $\fM$-primary ideal $\fa
\subseteq \cO_{X,O}$, let $D_{\lambda-\varepsilon}$ be the antinef
closure of $\lfloor(\lambda-\varepsilon) F - K_\pi\rfloor$ for
$\varepsilon >0$ small enough. Then we have
$\lambda=\lambda_{D_{\lambda-\varepsilon}}$ and the unique minimal
jumping divisor is $G_\lambda=G_{D_{\lambda-\varepsilon}}$.

\vskip 2mm

In general, a divisor $G\in\Lambda$  that contributes to the jumping
number $\lambda$ might not be contained in $G_\lambda$. For minimal
contributing divisors we have the following:

 \begin{proposition} \label{min_min}
 Let $\lambda$ be a jumping number  of an  ideal $\fa \subseteq \cO_{X,O}$ and $G_\lambda$ be its associated
 minimal jumping divisor. Then $G_D \leqslant G_\lambda$ for any antinef divisor $D\in {\rm Div}(X')$ such that $\lambda=\lambda_D$.
 \end{proposition}

 \begin{proof}
Let $D'$ be the antinef closure of $\lfloor\lambda F -
K_\pi\rfloor-G_D $.  By Proposition \ref{Prop3} we have $G_D=G_{D'}$
and $\lambda=\lambda_D=\lambda_{D'}$. Since the ideals
$\J(\A^{\lambda-\varepsilon}) \supseteq I_{D'}$ are nested, their
corresponding antinef divisors satisfy $D_{\lambda-\varepsilon}
\leqslant D'$ and they reach the same minimum
$\lambda_{D_{\lambda-\varepsilon}}=\lambda_{D'}=\lambda$. Hence,
$S_{D'} \subseteq S_{D_{\lambda-\varepsilon}}$ which implies
$G_D=G_{D'} \leqslant G_\lambda$ as we wanted.
 \end{proof}

\begin{corollary}\label{min_min2}
Let $\lambda$ be a jumping number of an  ideal $\fa \subseteq
\cO_{X,O}$. Then we have $G \leqslant G_{\lambda}$ for any critical
divisor $G$ associated to $\lambda$.
\end{corollary}

The reduced sum of all critical divisors equals the jumping divisor
$G_{\lambda}$ for simple complete ideals (see \cite[Thm. 2.3]{GM10}
for the smooth case). However this is no longer true in general.

\begin{example}  \label{ex_jump}
Let $X$ be a smooth surface and consider the $\fM$-primary ideal
$\fa \subseteq \cO_{X,O}$ whose dual graph is

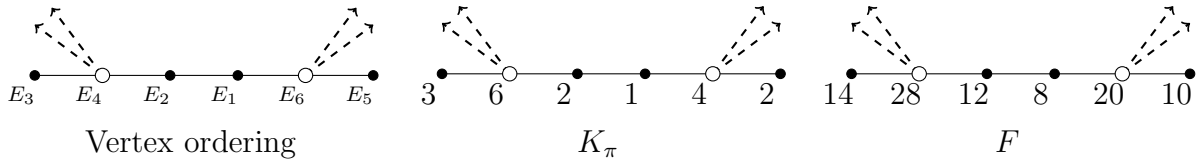
\begin{center}
\begin{tabular}{ccc}

\begin{tikzpicture}[scale=0.9]
   \draw  (-3,0) -- (2,0);
   \draw [dashed,->,thick] (1,0) -- (1.75,1);
   \draw [dashed,->,thick] (1,0) -- (2,0.75);
   \draw [dashed,->,thick] (-2,0) -- (-2.75,1);
   \draw [dashed,->,thick] (-2,0) -- (-3,0.75);
   \draw (-.2,-0.3) node {{\tiny $E_1$}};
   \draw (-1.2,-0.3) node {{\tiny $E_2$}};
   \draw (-3.2,-0.3) node {{\tiny $E_3$}};
   \draw (-2.2,-0.3) node {{\tiny $E_4$}};
   \draw (1.8,-0.3) node {{\tiny $E_5$}};
   \draw (0.8,-0.3) node {{\tiny $E_6$}};
   \filldraw  (0,0) circle (2pt)
              (-1,0) circle (2pt)
              (-3,0) circle (2pt)
              (2,0) circle (2pt);
   \filldraw  [fill=white]  (-2,0) circle (3pt)
              (1,0) circle (3pt);
 \end{tikzpicture}&
   \begin{tikzpicture}[scale=0.9]
   \draw  (-3,0) -- (2,0);
   \draw [dashed,->,thick] (1,0) -- (1.75,1);
   \draw [dashed,->,thick] (1,0) -- (2,0.75);
   \draw [dashed,->,thick] (-2,0) -- (-2.75,1);
   \draw [dashed,->,thick] (-2,0) -- (-3,0.75);
   \draw (-.2,-0.3) node {$1$};
   \draw (-1.2,-0.3) node {$2$};
   \draw (-3.2,-0.3) node {$3$};
   \draw (-2.2,-0.3) node {$6$};
   \draw (1.8,-0.3) node {$2$};
   \draw (0.8,-0.3) node {$4$};
   \filldraw  (0,0) circle (2pt)
              (-1,0) circle (2pt)
              (-3,0) circle (2pt)
              (2,0) circle (2pt);
   \filldraw  [fill=white]  (-2,0) circle (3pt)
              (1,0) circle (3pt);
 \end{tikzpicture}
  &
  \begin{tikzpicture}[scale=0.9]
   \draw  (-3,0) -- (2,0);
   \draw [dashed,->,thick] (1,0) -- (1.75,1);
   \draw [dashed,->,thick] (1,0) -- (2,0.75);
   \draw [dashed,->,thick] (-2,0) -- (-2.75,1);
   \draw [dashed,->,thick] (-2,0) -- (-3,0.75);
   \draw (-.2,-0.3) node {$8$};
   \draw (-1.2,-0.3) node {$12$};
   \draw (-3.2,-0.3) node {$14$};
   \draw (-2.2,-0.3) node {$28$};
   \draw (1.8,-0.3) node {$10$};
   \draw (0.8,-0.3) node {$20$};
   \filldraw  (0,0) circle (2pt)
              (-1,0) circle (2pt)
              (-3,0) circle (2pt)
              (2,0) circle (2pt);
   \filldraw  [fill=white]  (-2,0) circle (3pt)
              (1,0) circle (3pt);
 \end{tikzpicture}

\\
Vertex ordering & $K_\pi$& $F$
\end{tabular}
\end{center}

The multiplier ideals corresponding to the consecutive jumping
numbers $\frac{5}{7} < \frac{3}{4}$ are:

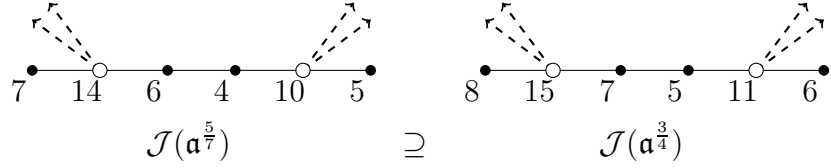
\begin{center}
\begin{tabular}{ccc}

\begin{tikzpicture}[scale=0.9]
   \draw  (-3,0) -- (2,0);
   \draw [dashed,->,thick] (1,0) -- (1.75,1);
   \draw [dashed,->,thick] (1,0) -- (2,0.75);
   \draw [dashed,->,thick] (-2,0) -- (-2.75,1);
   \draw [dashed,->,thick] (-2,0) -- (-3,0.75);
   \draw (-.2,-0.3) node {$4$};
   \draw (-1.2,-0.3) node {$6$};
   \draw (-3.2,-0.3) node {$7$};
   \draw (-2.2,-0.3) node {$14$};
   \draw (1.8,-0.3) node {$5$};
   \draw (0.8,-0.3) node {$10$};
   \filldraw  (0,0) circle (2pt)
              (-1,0) circle (2pt)
              (-3,0) circle (2pt)
              (2,0) circle (2pt);
   \filldraw  [fill=white]  (-2,0) circle (3pt)
              (1,0) circle (3pt);
 \end{tikzpicture}&
 &
  \begin{tikzpicture}[scale=0.9]
   \draw  (-3,0) -- (2,0);
   \draw [dashed,->,thick] (1,0) -- (1.75,1);
   \draw [dashed,->,thick] (1,0) -- (2,0.75);
   \draw [dashed,->,thick] (-2,0) -- (-2.75,1);
   \draw [dashed,->,thick] (-2,0) -- (-3,0.75);
   \draw (-.2,-0.3) node {$5$};
   \draw (-1.2,-0.3) node {$7$};
   \draw (-3.2,-0.3) node {$8$};
   \draw (-2.2,-0.3) node {$15$};
   \draw (1.8,-0.3) node {$6$};
   \draw (0.8,-0.3) node {$11$};
   \filldraw  (0,0) circle (2pt)
              (-1,0) circle (2pt)
              (-3,0) circle (2pt)
              (2,0) circle (2pt);
   \filldraw  [fill=white]  (-2,0) circle (3pt)
              (1,0) circle (3pt);
 \end{tikzpicture}
\\
$\J(\A^{\frac{5}{7}})$ & $\supseteq$ & $\J(\A^{\frac{3}{4}})$
\end{tabular}
\end{center}

The minimal jumping divisor corresponding to $\lambda=\frac{3}{4}$
is $G_{\frac{3}{4}}= E_1+E_2+E_4+E_6$ but the only critical divisors
are $E_4$ and $E_6$. In particular $$\J(\A^{\frac{5}{7}})
\varsupsetneq \pi_* \Oc_{X'}( \lceil K_{\pi} - \frac{3}{4} F\rceil +
E_4+E_6).$$ It is worth pointing out that
$$ \pi_* \Oc_{X'}(- D_{\frac{5}{7}} - E_4 - E_6)= \pi_* \Oc_{X'}(-D_{\frac{5}{7}} - G_{\frac{3}{4}}) = \J(\A^{\frac{3}{4}})$$
where $ D_{\frac{5}{7}}$ is the antinef closure of $ \lfloor
\frac{5}{7} F - K_{\pi} \rfloor$.  So minimality is not always
achieved for the divisor $G_\lambda$ in Proposition \ref{jump2}.
\end{example}

In general, not every nested ideal between two consecutive
multiplier ideals is given by a contributing divisor. The following
result identifies them precisely.

\begin{proposition}
Any nested ideal  $\J(\A^{\lambda-\varepsilon})\supseteq I_{D'}
\varsupsetneq   \J(\A^{\lambda})$ comes from a contributing divisor
$G$ associated to $\lambda$, i.e. $ I_{D'} =
\pi_*\Oc_{X'}\left(\left\lceil K_{\pi}- \lambda F\right\rceil +
G\right)$, if and only if $D'=D$ where $D$ is the antinef closure of
$\left\lfloor \lambda F - K_{\pi}\right\rfloor - G$ and in this case
$G=G_{D'}$.

\end{proposition}

\begin{proof}
Let $D'$ be the antinef closure of $\lfloor\lambda F -
K_\pi\rfloor-G $. By Proposition \ref{Prop3} we have $D\leqslant
D'$. On the other hand, Proposition \ref{G_DleqG} implies
$G_{D'}\leqslant G$ which gives
$$\lfloor\lambda F - K_\pi\rfloor-G \leqslant \lfloor\lambda F - K_\pi\rfloor-G_{D'}$$ and hence $D'\leqslant D$
so we get the desired result. The reverse implication is
straightforward.
\end{proof}

\begin{proposition}
Let $I_D$ be the ideal associated to an antinef divisor $D\in
\Lambda$. Then, $I_D$ is a multiplier ideal for the  ideal $\fa
\subseteq \cO_{X,O}$ if and only if $D$ is contained in the antinef
closure of $ \lfloor(\lambda_D -\varepsilon) F - K_\pi\rfloor$. If
this is the case, $D$ is also the antinef closure of
$\lfloor\lambda_D F - K_\pi\rfloor - G_D$.
\end{proposition}

\begin{proof}
By definition, we have  $ \lfloor(\lambda_D -\varepsilon) F -
K_\pi\rfloor \leqslant D$ because $\J(\A^{\lambda_D-\varepsilon})
\supseteq I_D$. We also have $I_D \not\subseteq \J(\A^{\lambda_D})$
so the only possibility for $I_D$ of being a multiplier ideal is
when $\J(\A^{\lambda_D-\varepsilon}) = I_D$ so, applying Lemma
\ref{Lem1}, $D$ must be contained in the antinef closure of$
\lfloor(\lambda_D -\varepsilon) F - K_\pi\rfloor$. The rest of the
statement follows from Theorem \ref{minimal_contributing}.
\end{proof}

%%%%%%%%%%%%%%%%%%%%%%%%%%

\end{document}